\documentclass[a4paper]{article}

\usepackage{lipsum}
\usepackage{amsfonts}
\usepackage{graphicx}
\usepackage{epstopdf}
\usepackage{algorithmic}
\usepackage{relsize}
\usepackage[]{mdframed}
\usepackage{tikz}

\usepackage{tabularx}
\usepackage{booktabs}

\usetikzlibrary{positioning}
\ifpdf
  \DeclareGraphicsExtensions{.eps,.pdf,.png,.jpg}
\else
  \DeclareGraphicsExtensions{.eps}
\fi

\usepackage{enumitem}
\usepackage{tikz}
\usepackage{tkz-euclide}
\setlist[enumerate]{leftmargin=.5in}
\setlist[itemize]{leftmargin=.5in}

\usepackage{amsmath,amsthm}
\usepackage{hyperref}
\usepackage{algorithm}
\usepackage{geometry}

\newtheorem{theorem}{Theorem}
\newtheorem{proposition}[theorem]{Proposition}
\newtheorem{remark}[theorem]{Remark}
\newtheorem{definition}[theorem]{Definition}
\newtheorem{lemma}[theorem]{Lemma}

\newcommand*\dx{\mathop{}\!\mathrm{d}}

\title{Exact Parameter Identification in PET Pharmacokinetic Modeling: Extension to the Reversible Two Tissue Compartment Model
}
\author{Martin Holler \thanks{IDea\_Lab - The Interdisciplinary Digital Lab at the University of Graz, Austria. MH further is a member of NAWI Graz (\href{https://www.nawigraz.at}{www.nawigraz.at}) and of BioTechMed Graz (\href{https://biotechmedgraz.at}{biotechmedgraz.at}) (\href{mailto:martin.holler@uni-graz.at}{martin.holler@uni-graz.at})} \and Erion Morina \thanks{Corresponding author. IDea\_Lab - The Interdisciplinary Digital Lab at the University of Graz, Austria.
		(\href{mailto:erion.morina@uni-graz.at}{erion.morina@uni-graz.at}).} \and Georg Schramm \thanks{Department of Imaging and Pathology, KU Leuven, Belgium. (\href{mailto:georg.schramm@kuleuven.be}{georg.schramm@kuleuven.be}).}}

\usepackage{amsopn}

\ifpdf
\hypersetup{
	pdftitle={Exact Parameter Identification in PET Pharmacokinetic Modeling: Extension to the Reversible Two Tissue Compartment Model},
	pdfauthor={M. Holler, E. Morina and G. Schramm}
}
\fi

\newcommand{\N}{\mathbb{N}}
\newcommand{\R}{\mathbb{R}}

\newcommand{\tis}{{\text{T}}}
\newcommand{\tot}{{\text{WB}}}
\newcommand{\art}{{\text{P}}}
\newcommand{\pet}{{\text{PET}}}

\newcommand{\free}{{\text{F}}}
\newcommand{\bd}{{\text{B}}}
\newcommand{\fbv}{{V_\bd}}

\newcommand{\contentskip}[1]{}
\renewcommand{\contentskip}[1]{#1}

\usepackage[belowskip=-2pt, skip=2pt]{caption}
\begin{document}
	\maketitle
	\begin{abstract}
		This paper addresses the problem of recovering tracer kinetic parameters from multi-region measurement data in quantitative PET imaging using the \emph{reversible} two tissue compartment model. Its main result is an extension of our previous work on the \emph{irreversible} two tissue compartment model. In analogy to our previous work, we show that also in the (practically highly relevant) reversible case, most tracer kinetic parameters can be uniquely identified from standard PET measurements (without additional full blood sample analysis that is usually performed in practice) and under reasonable assumptions. In addition, unique identifiability of \emph{all} parameters is shown provided that additional measurements from the (uncorrected) total arterial blood tracer concentration (which can be obtained from standard PET measurements or from a simple blood sample analysis) are available.
	\end{abstract}
	\begin{keywords}
		Quantitative PET imaging, exact reconstruction, two tissue compartment model
	\end{keywords}
	\begin{MSCcodes}
		65L09, 94A12, 92C55, 34C60
	\end{MSCcodes}
	\section{Introduction}
	
	Quantitative dynamic positron emission tomography (PET) images the space-time distribution of a radiotracer in tissue after injection. Depending on the choice of tracer, the reconstructed PET images allow conclusions about physiological parameters such as glucose metabolism, neuro receptor dynamics, blood flow, etc. The underlying pharmacokinetics are commonly modeled via compartment models which allow to model the dynamics between blood and tissue compartments using ordinary-differential-equations (ODEs). Exchange rates between compartments are defined via tracer kinetic parameters that define the response to the arterial tracer supply in a given tissue region. Their identification is usually based on measurements of the tracer concentration in tissue $C_\tis(t)$ and of the concentration of the original non-metabolized free tracer in the arterial blood plasma $C_\art(t)$ that is supplied to tissue. The former can be obtained from image-based measurements at voxel- or region-of-interest level. Estimating $C_\art(t)$, however, requires complicated and expensive blood sample analysis: While image-based measurements of the total arterial blood tracer concentration $C_\tot(t)$ are possible using image-derived input functions techniques, $C_\art(t)$ can be obtained from $C_\tot(t)$ only via an (unknown) attenuation factor $f(t)$ as $C_\art(t) = f(t)C_\tot(t)$, where the attenuation factor $f(t)$ accounts for (unknown) activity from radioactive molecules that are not available for exchange with tissue due to different metabolic mechanisms.

The goal of avoiding additional, expensive blood sample analysis for obtaining $C_\art(t)$ motivates identifying tissue kinetic parameters directly from the reconstructed PET images. For computational modeling approaches in quantitative PET see for instance \cite{dimit_2021, Ton15, van_der_weijden_2023, Ver13}. Other modeling approaches that additionally aim to simultaneously fit data from multiple anatomical regions coupled by common parameters include \cite{chen_2019,huesman_1997, matheson_2022, matheson_2023,raylman_1994,ogden_2015}. Specifically the extraction of the input function for dynamic  [\textsuperscript{18}F]-FDG data using the irreversible two tissue compartment model under joint analysis of multiple regions is studied in \cite{wong_2001} (SIME) and a corresponding empirical study in \cite{Ogden2010}. A recent work for the simultaneous estimation of kinetic parameters and input function for irreversible FDG modeling is \cite{Narciso_2024}. Yet another work in that direction, using the kernel method and taking into account small reversible effects, is \cite{Zhu_2024}. Finally, also noteworthy, is the work \cite{Sari2018} which considers SIME under image-based $C_\tot(t)$ and a specific model for $f(t)$.

From the mathematical perspective, obtaining kinetic tissue parameters from PET measurements is a highly non-linear inverse problem, specifically an ODE-based parameter identification problem, where the parameters need to be reconstructed from a low number of time-point measurements of a linear transformation of the state. While existing works such as the ones mentioned above focus on computational techniques for solving this problem, even in the idealized case of having noise-free measurements and measurements of the concentration of non-metabolized free tracer in the arterial blood plasma $C_\art(t)$, it is actually not even clear i) if the tissue parameters can be uniquely identified from the available measurements and ii) how many time-point measurements are required for such a unique identification.

To the best of our knowledge, our previous work \cite{Holler_2024} was the first one to address this uniqueness issue in quantitative PET analytically. Considering the \emph{reversible} two tissue compartment model, \cite{Holler_2024} shows that most of the kinetic parameters can in fact be recovered uniquely from image-based measurements of the tracer concentration in tissue $C_\tis(t)$ only, and provides also an explicit formula on the number of time-point-measurements that are sufficient for unique identifiability. Further, \cite{Holler_2024} proves that \emph{all} kinetic parameters can be recovered uniquely provided that additional measurements from the (uncorrected) total arterial blood tracer concentration are available, which can be obtained from standard PET measurements or from a simple blood sample analysis. In addition to these analytic results, \cite{Holler_2024} provides numerical experiments using a realistic simulation of the PET measurement process and noise that confirm the practical relevance and applicability of such uniqueness results.

A limitation of \cite{Holler_2024} is that it only considers the \emph{irreversible} two tissue compartment model, a model consisting of two tissue compartments (per voxel or region of interest) where, due to irreversibility, the concentration in one compartment can directly and explicitly obtained from the concentration in the other compartment. The uniqueness analysis of \cite{Holler_2024} heavily relies on this.

	Here, we consider a generalization of the analytic unique parameter identification results presented in \cite{Holler_2024} to the practically highly relevant case of the \emph{reversible} two tissue compartment model (see Figure \ref{fig:comparment_model_scheme} for a scheme) that is used to model the kinetics of most radiotracers,
	especially in neurological PET applications. While the underlying techniques of our proofs are similar to the ones of \cite{Holler_2024}, the extension to reversible compartments significantly complicates the proof and requires quite an extensive analysis.

An informal version of our main result is as follows:
	\begin{theorem}[Main result - informal version]
		Let $(K_1^i,k_2^i,k_3^i,k_4^i)$ be the kinetic parameters of different tissues $i=1,\ldots,n$ of the reversible two tissue compartment model, let $T$ be the number of time-points where PET measurements of $C_\tis^i(t)$ that are available, and let $p$ be the degree of a polyexponential parametrization used to approximate the unknown $C_\art(t)$.
		\begin{itemize}[leftmargin=0.7cm]
			\item If $T\geq2(p+4)$, and under some mild conditions as stated in Theorem \ref{thm:main_uniqueness}, the parameters $k_2^i,k_3^i, k_4^i$ for $i=1,\ldots,n$ can be identified uniquely already from the available image-based measurements of $C_\tis^i(t)$ in the different tissues $i$ without the need of $C_\tot(t)$ and $f(t)$.
			\item Further, the $K_1^i$ can also be identified already from these measurements up to a constant that is the same for all regions $i$.
			\item Moreover, the parameters $K_1^i$ can be identified exactly if a sufficient number of measurements of $C_\tot$ is available, without the need of $f(t)$.
		\end{itemize}
	\end{theorem}
	The precise result can be found in Theorem \ref{thm:main_uniqueness} below. The practical relevance of this result is that the kinetic parameters $(K_1^i,k_2^i,k_3^i,k_4^i)$, $i=1,\ldots,n$, can, in principle, be uniquely recovered (up to a global constant for the $K^i_1$ parameters) from image-based measurements of the tissue concentration in the different tissue types. For that, one requires sufficient quality of the reconstructed images at sufficiently many time-points (e.g. $T\geq 16$ if $p=4$). Ambiguity in $K_1^i$, $i=1,\ldots,n$, can be resolved provided enough high quality image-based measurements of the total arterial tracer concentration are available. These results can be generalized to the setup where PET image measurements consist of mixtures of blood tracer and tissue concentration rather than solely the latter (see Remark \ref{rem:cpet} below).
	
	\noindent\textbf{Scope of the paper.} In Section \ref{sec:compartment_model}, we introduce the reversible two tissue compartment model and discuss explicit solutions in both the general case and when the arterial concentration is parameterized by polyexponential functions. Section \ref{sec:unique_identifiability} is dedicated to our main results on the unique identifiability of parameters.

	\section{Model} \label{sec:compartment_model}
	The radiotracer is first administered to the blood system where it is available in non-metabolized, free form and determines the arterial plasma concentration $C_\art:[0,\infty) \rightarrow [0,\infty)$. Gradually the tracer is exchanged with tissue, resulting in a concentration in tissue $C_\tis:[0,\infty) \rightarrow [0,\infty)$. A common choice in practice to characterize the relationship between the concentration of the tracer in the blood (arterial blood plasma) and tissue (extra-vascular) compartment are two tissue compartment models. The tissue compartment is further subdivided into a \emph{free} and a \emph{bound} compartment, additionally depending on tissue $i=1,\dots, n$. The corresponding concentrations of the tracer in region $i$ are denoted by $C_\tis^i, C_\free^i, C_\bd^i:[0,\infty) \rightarrow [0,\infty)$ where $C_\tis^i=C_\free^i+C_\bd^i$. Here, as opposed to \cite{Holler_2024}, considering the irreversible two tissue compartment model, the radiotracer reaching the bound compartment can return to the free compartment. This extension is known as the \emph{reversible two tissue compartment model}. The interdependence of
	the different concentrations is described by the system of ordinary differential
	equations (ODEs):
	\begin{align*}
		\label{eq:odesystem}
		\begin{cases}
			\frac{\dx}{\dx t}C^i_\free= K^i_1 C_\art - \left(k^i_2+k^i_3\right) C^i_\free+k_4^i C_\bd^i, & t>0 \\
			\frac{\dx}{\dx t}C^i_\bd = k^i_3 C^i_\free-k_4^i C_\bd^i,                                    & t>0 \\
			C^i_\free\left(0\right)=0, ~ C^i_\bd\left(0\right)=0. \tag{S}
		\end{cases}
	\end{align*}
	The interactions between the (sub-)compartments for different anatomical regions $i$ are prescribed by the tracer kinetic parameters $K_1^i$, $k_2^i$, $k_3^i$ and $k_4^i$. See Figure \ref{fig:comparment_model_scheme} for an illustration of the model.
	
	\begin{figure}
		\begin{center}
			\begin{tikzpicture}[%
				>=stealth,
				node distance=3cm,
				on grid,
				auto
				]
				\draw[very thick] (0,0.5) rectangle (2,3.5);
				\draw[very thick] (3,0) rectangle (10.5,4);
				\draw[very thick] (3.5,0.5) rectangle (5.5,3.5);
				\draw[very thick] (7,0.5) rectangle (9,3.5);
				
				\path[->,very thick] (2,2) edge (3.5,2);
				\path[->,very thick] (5.5,2.25) edge (7,2.25);
				\path[->,very thick] (5.5,1.5) edge (7,-0.5);
				\path[->,very thick] (7,1.75) edge (5.5,1.75);
				
				\node (T) at (1,2)[align=center]{\(C_\art\)};;
				\node (T) at (4.5,2)[align=center]{\(C_\free\)};;
				\node (T) at (8,2)[align=center]{\(C_\bd\)};;
				\node (T) at (9.75,2)[align=center]{\(C_\tis\)};;
				
				\node (T) at (2.5,2.25)[align=center]{\(K_1\)};;
				\node (T) at (6.25,2.5)[align=center]{\(k_3\)};;
				\node (T) at (6.25,1.5)[align=center]{\(k_4\)};;
				\node (T) at (6.4,-0.5)[align=center]{\(k_2\)};;
				
			\end{tikzpicture}
		\end{center}
		\caption{\label{fig:comparment_model_scheme} Reversible two tissue compartment model. The (sub-)compartments are illustrated by boxes around the concentrations \(C_\art, C_\bd, C_\free\) and \(C_\tis\). The (directional) exchange rates \(K_1, k_2, k_3, k_4\) between the (sub-)compartments are represented by arrows.}
	\end{figure}
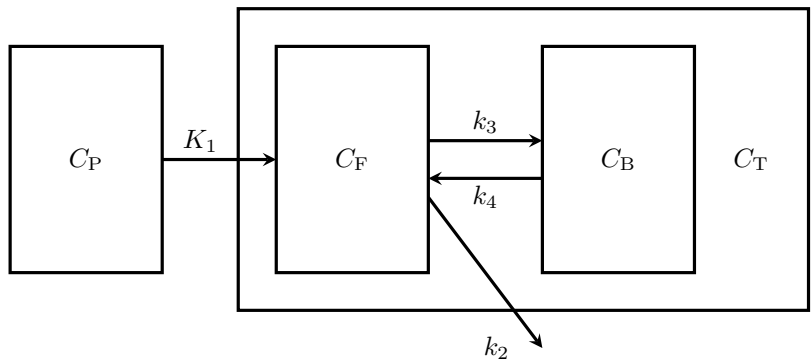

	Analogous to \cite{Holler_2024} we aim at studying identifiability of the parameters $K_1^i$, $k_2^i$, $k_3^i$ and $k_4^i$ for
	$i=1,\ldots,n$ based on measurements $C_\tis^i(t_l)$ at
	different time-points $t_1,\ldots,t_T$ and potentially on additional measurements related to $C_\art$.
	The concentration $C_\art$ is modeled as $C_\art(t) = f(t)C_\tot(t)$, where $C_\tot:[0,\infty) \rightarrow [0,\infty)$ is the total arterial tracer concentration and $f:[0,\infty) \rightarrow [0,1]$ with $f(0)=1$ is an attenuation term. While $C_\tot$ can be obtained from image-based measurements of arterial compartments or from a simple blood sample analysis, obtaining $f$ (and hence $C_\art$) would require a costly and time-consuming process of plasma separation and metabolite analysis in blood samples.

	In view of the parameter identification for ODE model
	\eqref{eq:odesystem} we parametrize the unknown arterial
	concentration $C_\art$ via
	polyexponential functions, i.e., for $p\in\mathbb{N}$ (degree), nonzero \((\lambda_i)_{i=1}^p\subset\mathbb{R}\) and pairwise distinct \((\mu_i)_{i=1}^p\subset \mathbb{R}\) (in practice negative as $C_\art$ decays to zero), $C_\art$ is given by
	\[
		C_\art\left(t\right) = \sum_{i=1}^p\lambda_i e^{\mu_i t}.
	\]
	Despite being standard in practice, this parametrization is also beneficial from a mathematical point of view since polyexponential functions can approximate continuous functions arbitrarily well on compacta (see \cite[Remark 3]{Holler_2024}). Furthermore, a parametrization of the arterial concentration already determines a parametrization of the resulting concentrations $C_\tis^i, C_\free^i, C_\bd^i$ of the ODE system \eqref{eq:odesystem} as we will see in Lemma \ref{lem:ode_solution_general} below. First we consider an auxiliary technical result, that will prove useful later.
	
	\begin{lemma}
		\label{cor:alphas}
		Let $k_2,k_3,k_4\geq 0$ and $k:=(k_2+k_3+k_4)/2$. Define further $$\alpha_1:=-k+\sqrt{k^2-k_2k_4}\quad \text{and} \quad \alpha_2:=-k-\sqrt{k^2-k_2k_4}.$$
		Then $\alpha_1,\alpha_2\in \mathbb{R}$ and the following identities hold true
		\begin{align}
			\label{alpha_identities}
		\alpha_1+\alpha_2+k_2+k_3+k_4=0 ~ ~ ~ \text{and} ~ ~ ~ \alpha_1\alpha_2 = k_2k_4.
		\end{align}
		If $(k_2+k_3+k_4)^2>4k_2k_4$ (which holds for $k_3>0$ or in case of $k_3=0$ for $k_2\neq k_4$) then $\alpha_2<\alpha_1$. Furthermore, if in addition to $k_3>0$ also $k_2>0$ and $k_4>0$ hold, then $$\alpha_2<-k_2<\alpha_1<0.$$
	\end{lemma}
	\begin{proof}
		The assertions of the lemma follow by immediate calculations.
	\end{proof}
	The next two results follow from standard ODE theory and provide explicitly the solutions $(C_\free^i,C_\bd^i)$ of the ODE system \eqref{eq:odesystem}. We start with the general case. 
	\begin{lemma} \label{lem:ode_solution_general}
		Let $C_\art:[0,\infty) \rightarrow [0,\infty)$ be continuous, and let the parameters $K_1^i$, $k_2^i, k_3^i, k_4^i\geq 0$ be fixed for $i=1,\ldots,n$ such that $(k_2^i+k_3^i+k_4^i)^2>4k_2^ik_4^i$. Then, for each $i=1,\ldots,n$, the ODE system \eqref{eq:odesystem} admits a unique solution $(C_\free^i,C_\bd^i)$ that is defined on all of $[0,\infty)$, and such that $C_\tis^i = C_\free^i + C_\bd^i$ is given by
		\begin{equation}
			\label{eq:ct_representation_general}
			C_\tis^i(t) = \frac{K_1^i}{\alpha_2^i-\alpha_1^i}\left[(\alpha_2^i+k_2^i)\int_0^t e^{-\alpha_1^i (s-t)}C_\art(s)\dx s-(\alpha_1^i+k_2^i)\int_0^t e^{-\alpha_2^i (s-t)}C_\art(s)\dx s\right]
		\end{equation}
		where $\alpha_1^i = -k^i+\sqrt{(k^i)^2-k_2^ik_4^i}$ and $\alpha_2^i = -k^i-\sqrt{(k^i)^2-k_2^ik_4^i}$ with $k^i=(k_2^i+k_3^i+k_4^i)/2$.
	\end{lemma}
	\begin{proof}
		Fix $i \in \{1,\ldots,n\}$. For $C = \begin{pmatrix}
			C_\free^i\\ C_\bd^i
		\end{pmatrix}$ it follows immediately from the equations in \eqref{eq:odesystem} that
		\begin{align}
			\label{C_integral}
			C(t) = \exp(At)\int_0^t \exp(-A s)\begin{pmatrix}
				K_1^i C_\art(s)\\ 0
			\end{pmatrix}\dx s\quad \text{for} ~ t\geq0 \quad \text{with} ~ A=\begin{pmatrix}
			-(k_2^i+k_3^i)& k_4^i\\
			k_3^i &-k_4^i
			\end{pmatrix}.
		\end{align}
		Consider first the case $k_4^i\neq 0$. Then, using the identities in \eqref{alpha_identities}, matrix $A$ is diagonalizable by
		\begin{align}
			\label{A_diagonalization}
			A = \underbrace{\frac{1}{k_4^i(\alpha_2^i-\alpha_1^i)}\begin{pmatrix}
				k_4^i&k_4^i\\
				\alpha_1^i+k_2^i+k_3^i&\alpha_2^i+k_2^i+k_3^i
			\end{pmatrix}}_{=:V}\underbrace{\begin{pmatrix}
			\alpha_1^i& 0\\
			0&\alpha_2^i
			\end{pmatrix}}_{=:D}\underbrace{\begin{pmatrix}
			\alpha_2^i+k_2^i+k_3^i & -k_4^i\\
			-(\alpha_1^i+k_2^i+k_3^i)&k_4^i
			\end{pmatrix}}_{=V^{-1}}
		\end{align}
		with eigenvalues $\alpha_1^i = -k^i+\sqrt{(k^i)^2-k_2^ik_4^i}$ and $\alpha_2^i = -k^i-\sqrt{(k^i)^2-k_2^ik_4^i}$ where $k^i=(k_2^i+k_3^i+k_4^i)/2$. Note that $\alpha_2^i\neq \alpha_1^i$ since $(k_2^i+k_3^i+k_4^i)^2>4k_2^ik_4^i$. By \eqref{C_integral} and \eqref{A_diagonalization} it holds true that
		\[
			C(t)=V \int_0^t \exp(D(t-s)) V^{-1}\begin{pmatrix}
				K_1^i C_\art(s)\\
				0
			\end{pmatrix}\dx s\quad \text{for} ~ t\geq0
		\]
		which implies after some calculations, again using the identities in \eqref{alpha_identities}, that
		\[
			C_\free^i(t) = \frac{K_1^i}{\alpha_2^i-\alpha_1^i}\left[(\alpha_2^i+k_2^i+k_3^i)\int_0^t e^{-\alpha_1^i (s-t)}C_\art(s)\dx s-(\alpha_1^i+k_2^i+k_3^i)\int_0^t e^{-\alpha_2^i (s-t)}C_\art(s)\dx s\right]
		\]
		\[
			\text{and} ~ ~ C_\bd^i(t)=\frac{K_1^ik_3^i}{\alpha_2^i-\alpha_1^i}\left[\int_0^t e^{-\alpha_2^i (s-t)}C_\art(s)\dx s-\int_0^t e^{-\alpha_1^i (s-t)}C_\art(s)\dx s\right]\quad \text{for} ~ t\geq 0
		\]
		and, finally, the assertion of the lemma due to $C_\tis^i = C_\free^i+C_\bd^i$ as claimed. If $k_4^i=0$ the equation for $C_\free^i$ in \eqref{eq:odesystem} immediately yields 
		\[
			C_\free^i(t)=K_1^i\int_0^te^{(k_2^i+k_3^i)(s-t)}C_\art(s)\dx s\quad \text{for} ~ t\geq 0
		\]
		which, under $k_2^i+k_3^i\neq 0$ by assumption as we are supposing $k_4^i=0$, implies that
		\[
			C_\bd^i(t)=-\frac{K_1^ik_3^i}{k_2^i+k_3^i}\int_0^te^{(k_2^i+k_3^i)(s-t)}C_\art(s)\dx s+\frac{K_1^ik_3^i}{k_2^i+k_3^i}\int_0^t C_\art(s)\dx s\quad \text{for} ~ t\geq 0
		\]
		and again the assertion of the lemma by $C_\tis^i = C_\free^i+C_\bd^i$ and Lemma \ref{cor:alphas}.
	\end{proof}
	In case $C_\art$ is modeled as a polyexponential function we obtain an explicit representation.
	\begin{lemma}
		\label{lem:solution_representation_polyexponential}
		Let $C_\art(t) = \sum_{j=1}^p \lambda_j e^{\mu_j t}$ for $(\lambda_j)_{j=1}^p, (\mu_j)_{j=1}^p\subset \mathbb{R}$. For $x,y\in \mathbb{R}$ denote $\mathbf{1}_{\left\{x\neq y\right\}}=1$ if $x\neq y$ and $\mathbf{1}_{\left\{x\neq y\right\}}=0$ if $x= y$.
		Then $(C_\tis^i)_{i=1}^n$ of Lemma
		\ref{lem:ode_solution_general} is given as
		\begin{multline}
			\label{ctextended}
			C_\tis^i(t) =\frac{K_1^i}{\alpha_2^i-\alpha_1^i}\sum_{j=1}^p\left(\frac{\alpha_2^i+k_2^i}{\mu_j-\alpha_1^i}\mathbf{1}_{\left\{\mu_j\neq \alpha_1^i\right\}}-\frac{\alpha_1^i+k_2^i}{\mu_j-\alpha_2^i}\mathbf{1}_{\left\{\mu_j\neq \alpha_2^i\right\}}\right)\lambda_je^{\mu_j t}\\
			+\left[K_1^i\frac{\alpha_1^i+k_2^i}{\alpha_2^i-\alpha_1^i}\sum_{\substack{j=1 \\ \mu_j\neq \alpha_2^i}}^p\frac{\lambda_j}{\mu_j-\alpha_2^i}\right]e^{\alpha_2^i t}-\left[K_1^i\frac{\alpha_2^i+k_2^i}{\alpha_2^i-\alpha_1^i}\sum_{\substack{j=1 \\ \mu_j\neq \alpha_1^i}}^p\frac{\lambda_j}{\mu_j-\alpha_1^i}\right]e^{\alpha_1^it}\\
			+\left[K_1^i\frac{\alpha_2^i+k_2^i}{\alpha_2^i-\alpha_1^i}\sum_{\substack{j=1 \\ \mu_j= \alpha_1^i}}^p\lambda_j\right]te^{\alpha_1^i t}-\left[K_1^i\frac{\alpha_1^i+k_2^i}{\alpha_2^i-\alpha_1^i}\sum_{\substack{j=1 \\ \mu_j= \alpha_2^i}}^p\lambda_j\right]te^{\alpha_2^it}\quad \text{for} ~ t\geq 0.
		\end{multline}
	\end{lemma}
	\begin{proof}
		This is an immediate consequence of \eqref{eq:ct_representation_general} in Lemma \ref{lem:ode_solution_general} and the representation of $C_\art$.
	\end{proof}

	\section{Unique identifiability} \label{sec:unique_identifiability}
	
	The result in Lemma \ref{lem:solution_representation_polyexponential} shows that the tracer concentration in tissue $C_\tis^i$ is given by a sum of polynomials (affine linear) scaled by exponentials in case the arterial plasma concentration $C_\art$ is polyexponential. Thus, a first crucial step towards unique identifiability of
	the parameters $K_1^i$, $k_2^i$, $k_3^i$ and $k_4^i$ from measurements of
	$C_\tis^i(t_l)$ at time points $t_1,\ldots,t_T$, analogous to \cite[Lemma 9]{Holler_2024}, is given by the following interpolation result. 

	\begin{lemma}[Unique interpolation]
		\label{lem:polyexpinter}
		Let \(m_1,\dots, m_p, T\in \mathbb{N}\) be such that \[
		2\left(m_1+\dots+m_p\right)\leq T.\]
		and \(\left(t_i, s_i\right)\in \mathbb{R}^2\),
		$i=1,\ldots,T$, with \(t_1< \dots< t_T\). Then for two functions $G, \tilde{G}$ of the form
		\[G(t) = \sum_{j=1}^p P_j\left(t\right)e^{\mu_j t} \quad \text{and}\quad \tilde G(t) = \sum_{j=1}^p \tilde P_j\left(t\right)e^{\tilde \mu_j t}
		\]
		with $P_j, \tilde{P}_j$ polynomials of degree $m_j-1$ for $1\leq j\leq p$, pairwise distinct $(\mu_j)_{j=1}^p \subset \mathbb{R}$ and pairwise distinct $(\tilde\mu_j)_{j=1}^p \subset \mathbb{R}$ it follows from the interpolation condition
		\begin{equation*}
			G(t_l) = s_l = \tilde{G}(t_l)
		\end{equation*}
		for $l=1,\ldots,T$ that $P_j \equiv \tilde{P_j}$ up to re-indexing,
		for all $j$ and $\mu_j = \tilde{\mu}_j$ for all $j$ where $P_j \not\equiv 0$.
	\end{lemma}
	\begin{proof}
		See the proof of \cite[Lemma 9]{Holler_2024}.
	\end{proof}
	The parameters that we aim to identify are summarized under the following notation.
	\begin{definition}[Parameter configuration] We call the parameters $p,n\in \N$, $((\lambda_j,\mu_j))_{j=1}^p \in \R^{2 \times p} $, $((K_1^i,k_2^i,k_3^i,k_4^i))_{i=1}^n \in \R^{4 \times n}_{\geq 0}$ together with the functions $(C^i_\tis)_{i=1}^n$ and
		\[ C_\art(t) = \sum_{j=1}^p \lambda_j e^{\mu_jt} \]
		a configuration of the reversible two tissue compartment model if $\lambda_j
		\neq 0$ for $j=1,\ldots,p$, the $\mu_j$, $j=1,\ldots,p$, are pairwise distinct,
		and, for $i=1,\ldots,n$, $C_\tis^i = C_\free^i + C_\bd^i$ with
		$(C_\free^i,C_\bd^i)$ the solution of the ODE system \eqref{eq:odesystem} with
		arterial concentration $C_\art$ and parameters $K_1^i,k_2^i,k_3^i,k_4^i$.

	\end{definition}
	
	We verify the unique identifiability result based on the following technical assumption on a parameter configuration $(p,n,((\lambda_j,\mu_j))_{j=1}^p
	,((K_1^i,k_2^i,k_3^i,k_4^i))_{i=1}^n,(C^i_\tis)_{i=1}^n,C_\art)$.
	\begin{mdframed}
	For any $1\leq j_0\leq p$, there are at least three different regions $1\leq i_s\leq n$, $s\in\left\{1,2,3\right\}$, where:
	\begin{equation} \label{ass:main_technical_ass} \tag{A}	
		\left\{\parbox{0.91\textwidth}{%
			\begin{itemize}[leftmargin=0.4cm]
				\item Each set $\left\{k^{i_s}_3+k_4^{i_s}\right\}_{s=1}^3$, $\left\{\alpha_1^{i_s}\right\}_{s=1}^3, \left\{\alpha_2^{i_s}\right\}_{s=1}^3$ has exactly three distinct elements
				\item It holds $\mu_{j_0} + k_3^{i_s}+k_4^{i_s} \neq 0$ for $s\in\left\{1,2,3\right\}$
				\item For each $s\in\left\{1,2,3\right\}$ either $\mu_{j_0} =\alpha_1^{i_s}$ or $\sum_{j: \mu_{j}\neq \alpha_1^{i_s}} \frac{\lambda_j}{\mu_j-\alpha_1^{i_s}} \neq 0$ applies
				\item For each $s\in\left\{1,2,3\right\}$ either $\mu_{j_0} =\alpha_2^{i_s}$ or $\sum_{j: \mu_{j}\neq \alpha_2^{i_s}} \frac{\lambda_j}{\mu_j-\alpha_2^{i_s}} \neq 0$ applies
			\end{itemize}%
		}\right.
	\end{equation}
\end{mdframed}
While Assumption \eqref{ass:main_technical_ass}, as required by our proof, is rather technical, the following lemma shows that a sufficient number of regions with different parameters (in a certain sense) is sufficient to ensure this assumption.
This makes sense also practically, since regions with similar parameters are essentially the same and, as can be seen from the solution formula of Lemma \ref{lem:solution_representation_polyexponential}, do not enrich the dynamics of the model at all.

	\begin{lemma} \label{lem:simpler_version_of_main_assumption}
		Assume that there are at least $p+3$ regions $i_{1},\ldots,i_{p+3}$, with $p\geq 1$, such that the set $\left\{k^{i_s}_3+k^{i_s}_4\right\}_{1\leq s\leq p+3}$ has exactly $p+3$ different elements and the set $\left\{\alpha_1^{i_s}\right\}_{1\leq s\leq p+3}\cup\left\{\alpha_2^{i_s}\right\}_{1\leq s\leq p+3}$ has exactly $2p+6$ different elements. Then Assumption \eqref{ass:main_technical_ass} holds true.
		\begin{proof}
			For $z \in \R$, note that
			\[
			\left( \prod_{\substack{j=1 \\ \ z \neq \mu_j}}^p {\mu_j-z}  \right)\left( \sum_{\substack{i=1 \\ z\neq \mu_i}}^p \frac{\lambda_i}{ \mu_i-z} \right) =
			\sum_{\substack{i=1 \\ z \neq \mu_i}}^p \lambda _i
			\prod_{\substack{j=1 \\ j \neq i \\ z \neq \mu_j}}^p (\mu_j-z)
			\]
			is a polynomial in $z$ of degree at most $p-1$. Hence, it can admit at most
			$p-1$ distinct roots. Now since there are at least $p+3$ regions where $\alpha_j^{i_s}$ are jointly pairwise distinct for $j\in\left\{1,2\right\}$, for at least
			four of them, say $i_1,\ldots,i_4$, $z\in\left\{\alpha_1^{i_s}\right\}_{1\leq s\leq 4}\cup \left\{\alpha_2^{i_s}\right\}_{1\leq s\leq 4}$ cannot be a root
			of the above polynomial. Further, for those four regions, since the $k^{i_s}_3+k_4^{i_s}$
			are pairwise distinct, for any given $\mu_{j_0}$, at most one region $i_s$ can be such that
			$\mu_{j_0} + k_3^{i_s}+k_4^{i_s} = 0$. As a consequence, the remaining three are such
			that the conditions of Assumption \eqref{ass:main_technical_ass} hold true.
		\end{proof}
	\end{lemma}
	Note that the above result also holds true if there are at least $2p+2$ regions $i_1,\dots, i_{2p+2}$ such that $\left\{\alpha_1^{i_s}\right\}_{1\leq s\leq 2p+2}$, $\left\{\alpha_2^{i_s}\right\}_{1\leq s\leq 2p+2}$ each have exactly $2p+2$ different elements. For the practical perspective, it is interesting to consider the result of Lemma \ref{lem:simpler_version_of_main_assumption} also from a probabilistic perspective:
		\begin{remark}[Assumption on different regions] \label{rem:probabilistic_interpretation}
		While it is of course not possible in practice to ensure the conditions of Lemma \ref{lem:simpler_version_of_main_assumption} (and hence Assumption \ref{ass:main_technical_ass}) on the ground truth parameter configuration, it is actually a condition that can be expected to hold already whenever the measurements comprise at least $p+3$ different regions: Assume that the metabolic tissue parameters are realizations of real-valued random variables that admit a density with respect to the standard Lebesgue measure (e.g., the parameters could be Gaussian distributed or even mixtures of Gaussians resulting from different patient classes in the population). Then, for each $K_1,k_2,k_3,k_4$, the probability that they admit a fixed, single value is zero, i.e., $\mathbb{P}(K_1=c)=0$ for any $c \in \R$ and similarly for $k_2$, $k_3$ and $k_4$. For two random samples of $k_3+k_4$, denoted by $k_3^1+k_4^1$ and $k_3^2+k_4^2$, it hence follows that
		\[
			\mathbb{P}((k_3^1+k_4^1)-(k_3^2+k_4^2)=0)=0,
		\]
		such that $p+3$ many realizations of $k_3+k_4$ are different almost surely. 
		Regarding the condition on $\alpha_1,\alpha_2$ we can argue similarly: Since the functions $\alpha^+$ and $\alpha^-$ given by
		\[
			\alpha^\pm: \mathbb{R}_+^3\to \mathbb{R}, ~ (k_2,k_3,k_4)\mapsto -(k_2+k_3+k_4)/2\pm\sqrt{(k_2+k_3+k_4)^2/4-k_2k_4}
		\]
		defining $\alpha_1,\alpha_2$ as in Lemma \ref{cor:alphas} are continuous, for random samples of $(k_2,k_3,k_4)$ denoted by $(k_2^1,k_3^1,k_4^1)$ and $(k_2^2,k_3^2,k_4^2)$ it again follows that
		\[
			\mathbb{P}(\alpha^+(k_2^1,k_3^1,k_4^1)-\alpha^+(k_2^2,k_3^2,k_4^2)=0)= 0 \quad \text{and} \quad \mathbb{P}(\alpha^+(k_2^1,k_3^1,k_4^1)-\alpha^-(k_2^2,k_3^2,k_4^2)=0)= 0
		\]
		such that $p+3$ many realizations of $\alpha_1, \alpha_2$ are jointly pairwise distinct almost surely.
		
		Thus, in summary, if the ground truth parameters of a specific patient result from a (reasonably distributed) random variable, the conditions of Lemma \ref{lem:simpler_version_of_main_assumption} (and hence Assumption \ref{ass:main_technical_ass}) can be expected to hold true whenever $p+3$ different regions are measured.
	\end{remark}	
	Next we show that the parameters of the ODE system \eqref{eq:odesystem} can be uniquely identified from time-discrete measurements $C_\tis^i(t_1), \ldots,C_\tis^i(t_T)$ with $i=1,\ldots,n$ and $C_\tot(s_1),\ldots,C_\tot(s_q)$, essentially under Assumption \eqref{ass:main_technical_ass} using the interpolation result in Lemma \ref{lem:polyexpinter}. In subsequent results we will infer interpretations on unique identifiability from a practical point of view.
	\begin{proposition} \label{prop:main_uniqueness_result}
		Let $(p,n,((\lambda_j,\mu_j))_{j=1}^p ,((K_1^i,k_2^i,k_3^i,k_4^i))_{i=1}^n,(C^i_\tis)_{i=1}^n,C_\art)$ be a configuration of the reversible two tissue compartment model with $p \geq 4$, $n \geq 3$, $K_1^i,k_2^i,k_3^i,k_4^i > 0$ for all $i=1,\ldots,n$ and such that Assumption \eqref{ass:main_technical_ass} holds. Let further $t_1,\ldots,t_T>0$ be different such that
		\[ T \geq 2(p+4) .\]
		Then, with $(\tilde p,n,((\tilde \lambda_j,\tilde \mu_j))_{j=1}^{\tilde p}
		,((\tilde K_1^i,\tilde k_2^i,\tilde k_3^i, \tilde k_4^i))_{i=1}^n,(\tilde
		C^i_\tis)_{i=1}^n,\tilde C_\art)$ any other configuration of the reversible
		two tissue compartment model such that $\tilde{p}\leq p$ and $\tilde k_3^i>0$ for $i=1,\dots, n$, it follows from
		$C^i_\tis(t_l) = \tilde C^i_\tis(t_l)$ for $l=1,\ldots, T$ and $i=1,\dots, n$, that
		\[ \tilde k_2^i = k_2^i,\quad \tilde k_3^i = k_3^i\quad \text{ and }\quad \tilde k_4^i = k_4^i
		\]
		for all $i=1,\ldots,n$, that there exists a constant $\zeta \neq 0$ such that
		\[ K_1^i = \zeta  \tilde  K_1^i \]
		for all $i=1,\ldots,n$, that $p= \tilde{p}$ and that (up to re-indexing)
		\[ \tilde \mu_j = \mu_j \text{ and } \tilde \lambda_j = \zeta \lambda_j \text{ for all }j=1,\ldots,p.
		\]
	\end{proposition}
	\begin{proof}
		Take $(p,n,((\lambda_j,\mu_j))_{j=1}^p ,((K_1^i,k_2^i,k_3^i, k_4^i))_{i=1}^n,(C^i_\tis)_{i=1}^n,C_\art)$ and $(\tilde p,n,((\tilde \lambda_j,\tilde \mu_j))_{j=1}^{\tilde p},$
		
		\noindent $((\tilde K_1^i,\tilde k_2^i,\tilde k_3^i, \tilde k_4^i))_{i=1}^n,(\tilde C^i_\tis)_{i=1}^n,\tilde C_\art)$ to be two configurations as stated in the proposition, with
		\begin{equation}
			\label{eq:uniquness_prop_measurment_equality}
			C_\tis^i(t_l) =  \tilde C_\tis^i(t_l)
		\end{equation}
		for $l=1,\ldots,T$ and $i=1,\dots,n$. Now the concentration in tissue  $C_\tis^i$ admits the
		representation:
			\begin{multline}
			\label{ctextended_mu_nonzero}
			C_\tis^i(t) =\frac{K_1^i}{\alpha_2^i-\alpha_1^i}\sum_{j=1}^p\left(\frac{\alpha_2^i+k_2^i}{\mu_j-\alpha_1^i}\mathbf{1}_{\left\{\mu_j\neq \alpha_1^i\right\}}-\frac{\alpha_1^i+k_2^i}{\mu_j-\alpha_2^i}\mathbf{1}_{\left\{\mu_j\neq \alpha_2^i\right\}}\right)\lambda_je^{\mu_j t}\\
			+\left[K_1^i\frac{\alpha_1^i+k_2^i}{\alpha_2^i-\alpha_1^i}\sum_{\substack{j=1 \\ \mu_j\neq \alpha_2^i}}^p\frac{\lambda_j}{\mu_j-\alpha_2^i}\right]e^{\alpha_2^i t}-\left[K_1^i\frac{\alpha_2^i+k_2^i}{\alpha_2^i-\alpha_1^i}\sum_{\substack{j=1 \\ \mu_j\neq \alpha_1^i}}^p\frac{\lambda_j}{\mu_j-\alpha_1^i}\right]e^{\alpha_1^it}\\
			+\left[K_1^i\frac{\alpha_2^i+k_2^i}{\alpha_2^i-\alpha_1^i}\sum_{\substack{j=1 \\ \mu_j= \alpha_1^i}}^p\lambda_j\right]te^{\alpha_1^i t}-\left[K_1^i\frac{\alpha_1^i+k_2^i}{\alpha_2^i-\alpha_1^i}\sum_{\substack{j=1 \\ \mu_j= \alpha_2^i}}^p\lambda_j\right]te^{\alpha_2^it}\quad \text{for} ~ t\geq 0.
		\end{multline}
	
		In particular for any region $i \in \{1,\ldots,n\}$, the coefficients of the $e^{\mu_j t}$ terms for $j=1,\ldots,p$ in this representation are given as either
		\begin{align}
			\label{coeff1}
		K_1^{i}\lambda_{j}\frac{\alpha_1^i+k_2^i}{(\alpha_2^i-\mu_j)(\alpha_2^i-\alpha_1^i)} \neq 0
		\end{align}
		in case $\mu_{j} =\alpha_1^i$, by
		\begin{align}
			\label{coeff2}
		K_1^{i}\lambda_{j}\frac{\alpha_2^i+k_2^i}{(\mu_j-\alpha_1^i)(\alpha_2^i-\alpha_1^i)} \neq 0
		\end{align}
		if $\mu_j=\alpha_2^i$, or, for $\mu_j\notin\left\{\alpha_1^i,\alpha_2^i\right\}$, using Lemma \ref{cor:alphas}, by
		\begin{align}
			\label{coeff3}
			K_1^i\lambda_j \frac{\mu_j+k_3^i+k_4^i}{(\mu_j-\alpha_1^i)(\mu_j-\alpha_2^i)}.
		\end{align}
		The latter can only be zero if $\mu_j + k_3^i+k_4^i=0$, which can happen for at most one $\mu_j$ by the $(\mu_j)_j$ being pairwise distinct.
		Since $p\geq 4$ by assumption, this implies in particular that $C_\tis^i$ is a
		nonzero function for any $i$. Also $\tilde C_\tis^i$ admits a representation as in
		\eqref{ctextended_mu_nonzero} by $\tilde k_3^i>0$, Lemma \ref{cor:alphas} and Lemma \ref{lem:ode_solution_general}. As a consequence of
		\eqref{eq:uniquness_prop_measurment_equality}, the condition \(T\geq 2(p+4)\) and the unique interpolation
		result of Lemma \ref{lem:polyexpinter}, this implies that $\tilde C_\tis^i$ is
		a nonzero function, such that in particular $\tilde K_1^i\neq 0$ for all $i\in
		\{1,\ldots,n\}$.
		
		As first step, we now aim to show that $\tilde{p} = p$ (in particular $\tilde{\lambda}_j\neq 0$ for all $j$) and that (up to re-indexing) $\mu_j=\tilde{\mu}_j$ for all $j=1,\ldots,p$. Note that by assumption $p\geq\tilde p$.
		
		\paragraph{Uniqueness of at least $p-3$ exponents $(\mu_j)_{j=1}^p$.}  We start with a region $i_0 \in \{1,\ldots,n\}$. In this region, as argued
		above, the coefficients of the $e^{\mu_j t}$ terms for $j\in\left\{1,\ldots,p\right\}$ can be zero for
		at most one $j^0\in\left\{1,\ldots,p\right\}$, i.e.,
		\begin{equation}
			\label{eq:muj0}
			\mu_{j^0}+k_3^{i_0}+k_4^{i_0}=0.
		\end{equation}
		Furthermore, there exist at most one $j^1\in\left\{1,\ldots,p\right\}$ and at most one $j^2\in\left\{1,\ldots,p\right\}$ with
		\begin{equation}
			\label{eq:muj12}
			\mu_{j^1}=\tilde\alpha_1^{i_0}\quad\text{ and }\quad \mu_{j^2}=\tilde\alpha_2^{i_0}.
		\end{equation}
		We argue first that $\tilde p\geq p-3$ and that at least $p-3$ of the $\mu_j, \tilde\mu_j$ coincide up to re-indexing.
		
		In case $\mu_j\notin\left\{\alpha_1^{i_0},\alpha_2^{i_0}\right\}$ for all $j \notin \{j^0, j^1,j^2\}$ it holds that the coefficients of the $e^{\mu_j t}$ are nonzero for $j\notin \left\{j^0, j^1,j^2\right\}$. Then the unique interpolation
		result of Lemma \ref{lem:polyexpinter} applied to $C_\tis^{i_0}$ and $\tilde
		C_\tis^{i_0}$ yields that $\tilde{p}\geq p-3 \geq 1$ and that (up to
		re-indexing) $\mu_j=\tilde{\mu}_j$ for all $j \notin \{j^0, j^1,j^2\}$.
		
		To show that the other case of $\mu_j\in\left\{\alpha_1^{i_0},\alpha_2^{i_0}\right\}$ for $j\notin\left\{j^0,j^1,j^2\right\}$ cannot occur, it suffices to show that an index $\hat{j}$ with $\mu_{\hat{j}}\in\left\{\alpha_1^{i_0},\alpha_2^{i_0}\right\}$ necessarily fulfills $\hat{j}\in\left\{j^1,j^2\right\}$. For that let w.l.o.g.  $\mu_{\hat{j}}=\alpha_1^{i_0}$. Then the representation in \eqref{ctextended_mu_nonzero} implies that the coefficient of $te^{\alpha_1^{i_0}t}$ is nonzero and by Lemma \ref{lem:polyexpinter} the coefficient of either $te^{\tilde\alpha_1^{i_0}t}$ or $te^{\tilde\alpha_2^{i_0}t}$ is nonzero too. This implies that $\mu_{\hat{j}}\in\left\{\tilde\alpha_1^{i_0},\tilde\alpha_2^{i_0}\right\}$ and hence, by \eqref{eq:muj12} that either $\hat{j}=j^1$ or $\hat{j}=j^2$ since the $(\mu_j)_j$ are pairwise distinct. %
		
		\paragraph{Uniqueness of the remaining exponents $(\mu_j)_{j=1}^p$.} We proceed by showing that $\tilde{p}=p$ and that also the remaining $\mu_{j^0}, \mu_{j^1}, \mu_{j^2}$ necessarily coincide with some $\tilde\mu_j$, respectively. As a consequence of Assumption \eqref{ass:main_technical_ass}, we can pick a
		region $i_1\neq i_0$ with 
		\begin{equation}
			\label{eq:new_region}
			k_3^{i_1}+k_4^{i_1} \neq k_3^{i_0}+k_4^{i_0}, ~ \mu_{j^1}+k_3^{i_1}+k_4^{i_1}\neq 0, ~ \mu_{j^2}+k_3^{i_1}+k_4^{i_1}\neq 0 ~ \text{ and } ~ \mu_{j^0} + k_3^{i_1}+k_4^{i_1}\neq 0,
		\end{equation}
		where the latter inequality holds since already $\mu_{j^0} + k_3^{i_0}+k_4^{i_0}=0$ by \eqref{eq:muj0}. This means that the coefficients of $e^{\mu_{j^0}t}, e^{\mu_{j^1}t}$ and $e^{\mu_{j^2}t}$ in the representation of
		$C_\tis^{i_1}$ as in \eqref{ctextended} are nonzero (see \eqref{coeff1}-\eqref{coeff3}). Again by the $(\mu_j)_j$
		being pairwise distinct, this implies that $\tilde{p}\geq p-2$ and that (up to
		re-indexing) either $\mu_{j^0} = \tilde \mu_{j^0}$ or $\mu_{j^1} =
		\tilde \mu_{j^1}$ or $\mu_{j^2} = \tilde \mu_{j^2}$. If all three equalities apply there is nothing left to show in this step. We consider the case distinction of exactly two of these equalities holding true (\nameref{mu:case1}) or only one holding true (\nameref{mu:case2}) which is possible in this reduced form as the roles of $\mu_{j^0},\mu_{j^1}, \mu_{j^2}$ are interchangeable.\vspace{-0.4cm}   %
		\paragraph{Case I.}\label{mu:case1} W.l.o.g. suppose that $\mu_{j^1} =
		\tilde \mu_{j^1}$ and $\mu_{j^2} = \tilde \mu_{j^2}$ but $\mu_{j^0}$ is different from any $\tilde \mu_j$. Now as a consequence of Assumption \eqref{ass:main_technical_ass} we can pick $i_2,i_3$ and $i_4$ to be regions where the $\left\{\alpha_1^{i_s}\right\}_{2\leq s\leq 4}$ are pairwise distinct, $\left\{\alpha_2^{i_s}\right\}_{2\leq s\leq 4}$ are pairwise distinct, and for $s\in\left\{2,3,4\right\}$ it holds $\mu_{j^0}+k_3^{i_s}+k_4^{i_s}\neq 0$ together with either $\mu_{j^0}=\alpha_i^{i_s}$ or the coefficient of $e^{\alpha_i^{i_s}t}$ in the representation of $C_\tis^{i_s}$ being nonzero for $i\in\left\{1,2\right\}$, respectively.
		
		\noindent \textbf{Case I.A} If there exist some $s\in\left\{2,3,4\right\}$ and $i\in\left\{1,2\right\}$ with $\mu_{j^0}=\alpha_i^{i_s}$ we derive that the coefficient of $te^{\alpha_i^{i_s}t}$ in the representation of $C_\tis^{i_s}$ is nonzero. Due to the unique interpolation result in Lemma \ref{lem:polyexpinter} the same necessarily holds true for $\tilde C_\tis^{i_s}$ either for $t e^{\tilde \alpha_1^{i_s}t}$ or $te^{\tilde \alpha_2^{i_s}t}$ with $\alpha_i^{i_s}\in\left\{\tilde\alpha_1^{i_s},\tilde\alpha_2^{i_s}\right\}$. This is only possible if either $\tilde \alpha_1^{i_s}$ or $\tilde \alpha_2^{i_s}$ (and thus, also $\alpha_i^{i_s}$) coincides with some $\tilde \mu_j$ which hence, necessarily coincides with $\mu_{j^0}$. This, however, is a contradiction, since by assumption $\mu_{j^0}$ is different from any $\tilde\mu_j$.
		
		\noindent \textbf{Case I.B} Suppose that $\mu_{j^0}\notin \left\{\alpha_1^{i_s},\alpha_2^{i_s}\right\}$ for all $s\in\left\{2,3,4\right\}$. As the coefficient of $e^{\mu_{j^0}t}$ is nonzero for each region $i_2, i_3, i_4$ and $\mu_{j^0}$ is different from the $\tilde \mu_j$ it holds true that 
		\begin{equation}
			\label{eq1:mu}
			\mu_{j^0}\in \left\{\tilde \alpha_1^{i_s},\tilde \alpha_2^{i_s}\right\}
		\end{equation} 
		for each $s\in\left\{2,3,4\right\}$. As $\mu_{j^0}\notin \left\{\alpha_1^{i_s},\alpha_2^{i_s}\right\}$, the coefficients of the $e^{\alpha_1^{i_s}t}$ and $e^{\alpha_2^{i_s}t}$ in the representations of $C_\tis^{i_s}$ are nonzero for $s\in\left\{2,3,4\right\}$ by Assumption \eqref{ass:main_technical_ass} at the beginning of \nameref{mu:case1}.\\
		For fixed $s\in\left\{2,3,4\right\}$ both $\alpha_1^{i_s}$ and $\alpha_2^{i_s}$ cannot coincide with some $\mu_j$, respectively. The reason is that in this case, due to the interpolation result in Lemma \ref{lem:polyexpinter} and Lemma \ref{cor:alphas}, it would hold true that $\alpha_i^{i_s}=\tilde \alpha_i^{i_s}$ for $i\in\left\{1,2\right\}$, which in turn would imply that both $\tilde\alpha_1^{i_s}$ and $\tilde\alpha_2^{i_s}$ coincide with some $\tilde \mu_j$, respectively, and hence, also $\mu_{j_0}$ by \eqref{eq1:mu}, in contradiction to $\mu_{j^0}$ being different from any $\tilde \mu_{j}$.\\
		For fixed $s\in\left\{2,3,4\right\}$ if only $\alpha_1^{i_s}$ coincides with some $\mu_j$ then w.l.o.g. we derive $\alpha_1^{i_s}=\tilde \alpha_1^{i_s}$. The coefficient of $e^{\alpha_2^{i_s}t}$ being nonzero, by Assumption \eqref{ass:main_technical_ass}, implies that $\alpha_2^{i_s}=\tilde \mu_{j^0}$ and $\tilde p = p$ (as $\alpha_2^{i_s}=\tilde \mu_j=\mu_j$ for $j\neq j^0$ is not possible by assumption and $\alpha_2^{i_s}=\tilde \alpha_2^{i_s}$ leads to a contradiction as previously using \eqref{eq1:mu} and that $\mu_{j^0}\notin \left\{\alpha_1^{i_s},\alpha_2^{i_s}\right\}$). A similar conclusion can be derived for $\alpha_1^{i_s}$ instead of $\alpha_2^{i_s}$ above such that we obtain $\tilde \mu_{j^0}=\alpha_1^{i_s}$ or $\tilde \mu_{j^0}=\alpha_2^{i_s}$.\\
		If none of the $\alpha_i^{i_s}$ for $i\in\left\{1,2\right\}$ coincides with any $\mu_j$,  we argue again that $\tilde p=p$ and $\tilde \mu_{j^0}\in \left\{\alpha_1^{i_s},\alpha_2^{i_s}\right\}$. Certainly, since $\tilde \mu_j=\mu_j$ for $j\neq j^0$ it would follow in case $\tilde \mu_{j^0}\notin \left\{\alpha_1^{i_s},\alpha_2^{i_s}\right\}$ by the coefficients of $e^{\alpha_i^{i_s}t}$ being nonzero for $i\in\left\{1,2\right\}$ that $\left\{\tilde \alpha_1^{i_s},\tilde \alpha_2^{i_s}\right\}=\left\{\alpha_1^{i_s},\alpha_2^{i_s}\right\}$. By \eqref{eq1:mu} this implies a contradiction to $\alpha_i^{i_s}$ being different to any $\mu_j$ for $i\in\left\{1,2\right\}$.
		
		In either case one of the equalities $\tilde \mu_{j^0}=\alpha_1^{i_s}$ or $\tilde \mu_{j^0}=\alpha_2^{i_s}$ is necessarily attained for at least two $s\in\left\{2,3,4\right\}$ and we get a contradiction to the $\alpha_i^{i_s}$ being different for $s\in\left\{2,3,4\right\}$, each for $i\in\left\{1,2\right\}$, respectively.
		
		Hence, \nameref{mu:case1} is not possible.\vspace{-0.4cm}
		
		\paragraph{Case II.}\label{mu:case2}  W.l.o.g. suppose that $\mu_{j^2} = \tilde \mu_{j^2}$ but $\mu_{j^0},\mu_{j^1}$ are different from any $\tilde \mu_j$, respectively. Employing Assumption \eqref{ass:main_technical_ass} yields the existence of regions $i_2, i_3$ where the $\alpha_1^{i_s}, \alpha_2^{i_s}$ are each pairwise distinct and $\mu_{j}+k_3^{i_s}+k_4^{i_s}\neq 0$ for $j\in \left\{j^0, j^1\right\}$ and $s\in\left\{2,3\right\}$. Furthermore, either $\mu_{j^0}=\alpha_i^{i_s}$ or the coefficient of $e^{\alpha_i^{i_s}t}$ is nonzero in the representation of $C_\tis^{i_s}$, for $i\in\left\{1,2\right\}$.
		
		Neither $\mu_{j^0}$ nor $\mu_{j^1}$ can coincide with $\alpha_i^{i_s}$ for some $i\in\left\{1,2\right\}$. This can be seen exemplarily for $\mu_{j^0}=\alpha_1^{i_s}$ as follows: The coefficient of $te^{\alpha_1^{i_s}t}$ in \eqref{ctextended_mu_nonzero} would be nonzero, implying that the coefficient of $te^{\tilde\alpha_i^{i_s}t}$ is nonzero and $\mu_{j^0}=\tilde\alpha_i^{i_s}$ for some $i\in\left\{1,2\right\}$ by the unique interpolation result in Lemma \ref{lem:polyexpinter}. The coefficient of $te^{\tilde\alpha_i^{i_s}t}$ being nonzero would imply that $\tilde\alpha_i^{i_s}$ and hence, also $\mu_{j^0}$ coincides with some $\tilde\mu_j$ which is a contradiction to $\mu_{j^0}$ being different from any $\tilde\mu_j$.
		
		As a consequence, the coefficients of $e^{\alpha_i^{i_s}t}$ are nonzero for $i\in\left\{1,2\right\}$. Since the coefficients of $e^{\mu_{j^0}t}, e^{\mu_{j^1}t}$ are nonzero it necessarily holds true that 
		\begin{align}
			\label{eq2:mu}
			\left\{\mu_{j^0}, \mu_{j^1}\right\}=\left\{\tilde \alpha_1^{i_s},\tilde \alpha_2^{i_s}\right\}.
		\end{align}
		Thus, it holds $\alpha_i^{i_s}\notin\left\{\tilde\alpha_1^{i_s},\tilde\alpha_2^{i_s}\right\}$ since we have previously argued that neither $\mu_{j^0}$ nor $\mu_{j^1}$ can coincide with $\alpha_i^{i_s}$ for some $i\in\left\{1,2\right\}$. We argue that $\alpha_i^{i_s}$ also cannot equal none of the $\tilde \mu_j=\mu_j$ for $j\notin\left\{j^0, j^1\right\}$. Exemplarily, if $\alpha_1^{i_s}$ equals some $\mu_j$ for $j\notin\left\{j^0,j^1\right\}$, then the coefficient of $te^{\alpha_1^{i_s}t}$ would be nonzero in \eqref{ctextended_mu_nonzero}, implying that the coefficient of $te^{\tilde\alpha_i^{i_s}t}$ is nonzero and $\mu_{j}=\tilde \alpha_i^{i_s}$ for some $i\in\left\{1,2\right\}$ by the unique interpolation result in Lemma \ref{lem:polyexpinter}. By \eqref{eq2:mu} this is a contradiction to $\mu_j\notin\left\{\mu_{j^0}, \mu_{j^1}\right\}$ for $j\notin\left\{j^0,j^1\right\}$ since the $(\mu)_j$ are pairwise distinct. Hence, it must hold $\tilde p = p$ and $\left\{\tilde \mu_{j^0}, \tilde \mu_{j^1}\right\}=\left\{\alpha_1^{i_s}, \alpha_2^{i_s}\right\}$ for $s=2,3$. As $\alpha_2^{i_s}<\alpha_1^{i_s}$ by Lemma \ref{cor:alphas} this is again a contradiction to the $\alpha_i^{i_s}$ being pairwise distinct, for $i$, respectively. Thus, also \nameref{mu:case2} is not possible. We conclude that $\tilde p = p$ and $\mu_j = \tilde \mu_j$ for all $j=1,\dots, p$ up to reindexing.\\
		
		Note that in case one of the $j_0$, $j_1$ or $j_2$ in \eqref{eq:muj0}-\eqref{eq:muj12} does not exist, the argumentation on uniqueness of the remaining exponents $(\mu_j)_{j=1}^p$ reduces to \nameref{mu:case1} above. If at least two of the $j_s$ for $s\in \left\{0,1,2\right\}$ with \eqref{eq:muj0}-\eqref{eq:muj12} do not exist, uniqueness follows directly by the paragraph before \nameref{mu:case1} since \nameref{mu:case1} and \nameref{mu:case2} cannot occur.
		
		\paragraph{Uniqueness of $\alpha_1^i, \alpha_2^i$ for at least three regions.}\label{para:unique_3alpha} Let $j_0 \in \{1,\ldots,p\}$ and let $i_0$ be any region such that either $\mu_{j_0} = \alpha_i^{i_0}$ (i.e., the coefficient of $t e^{\alpha_i^{i_0}t}$ in the representation of $C_\tis^{i_0}$ is nonzero) or the coefficient of $e^{\alpha_i^{i_0}t}$ in the representation of $C_\tis^{i_0}$ is nonzero for $i\in\left\{1,2\right\}$. Note that, according to Assumption \eqref{ass:main_technical_ass}, at least three such regions exist.
		
		\noindent \textbf{Case I.} Assume that $\mu_{j_0} = \alpha_1^{i_0}$. This implies that the coefficient of $te^{\alpha_1^{i_0}t}$ is nonzero and, consequently, by the unique interpolation result in Lemma \ref{lem:polyexpinter}, that $\alpha_1^{i_0}\in \left\{\tilde \alpha_1^{i_0},\tilde \alpha_2^{i_0}\right\}$. As $\mu_{j_0}\neq \alpha_2^{i_0}$ it necessarily holds true that the coefficient of $e^{\alpha_2^{i_0}t}$ is nonzero and $\alpha_2^{i_0}$ coincides either with some $\tilde \mu_j$ or $\tilde \alpha_i^{i_0}$. In the former case we derive that it coincides with some $\mu_j=\tilde \mu_j$ and as previously, $\alpha_2^{i_0}\in \left\{\tilde \alpha_1^{i_0}, \tilde \alpha_2^{i_0}\right\}$. In any case we obtain that $\left\{\alpha_1^{i_0},\alpha_2^{i_0}\right\}=\left\{\tilde \alpha_1^{i_0}, \tilde \alpha_2^{i_0}\right\}$ and finally, that $\alpha_i^{i_0}=\tilde \alpha_i^{i_0}$ for $i\in\left\{1,2\right\}$, since  by Lemma \ref{cor:alphas}, $\alpha_2^{i_0}<\alpha_1^{i_0}$ and $\tilde\alpha_2^{i_0}<\tilde\alpha_1^{i_0}$. Considering $\alpha_2^{i_0}$ instead of $\alpha_1^{i_0}$ at the beginning of Case I yields the same result.
		
		\noindent \textbf{Case II.} If $\mu_{j^0} \neq \alpha_i^{i_0}$ for $i\in\left\{1,2\right\}$ then the coefficients of $e^{\alpha_1^{i_0}t}$, $e^{\alpha_2^{i_0}t}$ are nonzero, and  $\alpha_1^{i_0}, \alpha_2^{i_0}$ necessarily match some exponents in the representation of $\tilde C^{i_0}_\tis$. In case $\alpha_i^{i_0}$ matches some $\tilde{\mu}_j = \mu_j $ we derive as previously that $\alpha_i^{i_0}\in \left\{\tilde \alpha_1^{i_0},\tilde \alpha_2^{i_0}\right\}$. Otherwise this inclusion follows immediately by Lemma \ref{lem:polyexpinter}. Hence again $\left\{\alpha_1^{i_0},\alpha_2^{i_0}\right\}=\left\{\tilde \alpha_1^{i_0}, \tilde \alpha_2^{i_0}\right\}$ and as previously $\alpha_i^{i_0}=\tilde \alpha_i^{i_0}$ for $i\in\left\{1,2\right\}$ follows by Lemma \ref{cor:alphas}.
		
		\paragraph{Uniqueness of $k_2^i,k_3^i,k_4^i$ for at least three regions.}\label{para:unique_3k}
		First note that for any $i \in \{1,\ldots,n\}$ where $\alpha_1^i=\tilde \alpha_1^i, \alpha_2^i=\tilde \alpha_2^i$, from the unique interpolation result in Lemma \ref{lem:polyexpinter}, it follows that
		\begin{equation} \label{eq:coefficient_basic_equality}
			K^{i}_1 \lambda_j(\mu_j + k^ {i}_3+k_4^i) = \tilde K^{i}_1 \tilde \lambda_j(\mu_j + \tilde k^ {i}_3+\tilde k_4^i)
		\end{equation}
		for all $j\in\left\{1,\ldots,p\right\}$: Indeed, in case of $j\in\left\{1,\ldots,p\right\}$ with $\mu_j=\alpha_1^i$, it follows from the coefficients of $t e^{\alpha_1^it}$ in $C_\tis^i$ and $\tilde C_\tis^i$ (see \eqref{ctextended_mu_nonzero}) being equal that
		\[ K_1^{i} (\alpha_2^i+k_2^{i})\lambda_j = \tilde K_1^{i} (\tilde \alpha_2^i+\tilde k_2^{i})\tilde \lambda_j,
		\]
		which, employing Lemma \ref{cor:alphas} and using $\alpha_1^i=\mu_j$ and $\alpha_1^i=\tilde \alpha_1^i$, implies \eqref{eq:coefficient_basic_equality} as claimed. If $\mu_j=\alpha_2^i$, it follows from the coefficients of $t e^{\alpha_2^it}$ in $C_\tis^i$ and $\tilde C_\tis^i$ (see \eqref{ctextended_mu_nonzero}) being equal that
		\[ K_1^{i} (\alpha_1^i+k_2^{i})\lambda_j = \tilde K_1^{i} (\tilde \alpha_1^i+\tilde k_2^{i})\tilde \lambda_j,
		\]
		which, employing Lemma \ref{cor:alphas} and using $\alpha_2^i=\mu_j$ and $\alpha_2^i=\tilde \alpha_2^i$, implies \eqref{eq:coefficient_basic_equality} as claimed. In the remaining case that $\mu_j\notin\left\{\alpha_1^i,\alpha_2^i\right\}$, the equality \eqref{eq:coefficient_basic_equality} follows directly from \eqref{coeff3} since the coefficients of $e^{\mu_j t}$ in $C_\tis^i$ and $e^{\tilde \mu_j t}$ in $\tilde
		C_\tis^i$ coincide as $\mu_j=\tilde\mu_j$.
		
		Now let $i_0$ be any region where $\alpha_1^{i_0}=\tilde \alpha_1^{i_0}, \alpha_2^{i_0}=\tilde \alpha_2^{i_0}$, and for which we want to show that $k_2^{i_0} = \tilde
		k_2^{i_0} $, $k_3^{i_0} = \tilde k_3^{i_0} $ and $k_4^{i_0} = \tilde k_4^{i_0} $. Again we consider several
		cases.\vspace{-0.4cm}
		
		\paragraph{Case I.}\label{para:unique_k_case1} Assume that there exists $j_0 \in \{1,\ldots,p\}$ such that $ \mu_{j_0} + k_3^{i_0}+k_4^{i_0}= 0$. In this case, it follows from \eqref{eq:coefficient_basic_equality} that also $\mu_{j_0} + \tilde k_3^{i_0}+ \tilde k_4^{i_0} = 0$ (note that $\tilde \lambda_{j_0} \neq 0 $ and $\tilde K_1^{i_0} \neq 0$ since $\tilde{p}=p$), hence 
		\begin{equation}
			\label{equality_k3_k4}
			k_3^{i_0}+k_4^{i_0} = \tilde k_3^{i_0}+\tilde k_4^{i_0}.
		\end{equation}
		 As a consequence of the first identity in \eqref{alpha_identities} of Lemma \ref{cor:alphas} it holds $k_2^{i_0}=\tilde k_2^{i_0}$. It also holds $k_4^{i_0} = \tilde k_4^{i_0}$ as we have for $k^{i_0} = (k_2^{i_0}+k_3^{i_0}+k_4^{i_0})/2=(\tilde k_2^{i_0}+\tilde k_3^{i_0}+\tilde k_4^{i_0})/2=\tilde k^{i_0}$ that
		 \[
		 	-k^{i_0}+\sqrt{(k^{i_0})^2-k_2^{i_0}k_4^{i_0}}=\alpha_1^{i_0}=\tilde\alpha_1^{i_0}=-\tilde k^{i_0}+\sqrt{(\tilde k^{i_0})^2-\tilde k_2^{i_0} \tilde k_4^{i_0}}=	-k^{i_0}+\sqrt{(k^{i_0})^2-k_2^{i_0}\tilde k_4^{i_0}}
		 \]
		 and finally, by \eqref{equality_k3_k4} and $k_2^{i_0}>0$ by assumption of the Proposition, also $k_3^{i_0}=\tilde k_3^{i_0}$.
		
		\noindent \textbf{Case II.} Assume that $ \mu_{j} + k_3^{i_0}+k_4^{i_0} \neq 0$ for all $j$. In this case, using Assumption \eqref{ass:main_technical_ass} and the considerations on \nameref{para:unique_3alpha}, we can select $i_1$ to be a second region where again $\alpha_1^{i_1}=\tilde \alpha_1^{i_1}, \alpha_2^{i_1}=\tilde \alpha_2^{i_1}$ and such that $k_3^{i_0}+k_4^{i_0} \neq k_3^{i_1}+k_4^{i_1}$. We have two cases.
		
		\noindent \textbf{Case II.A} Assume that there exists $j_1 \in \{1,\ldots,p\}$ such that $ \mu_{j_1} + k_3^{i_1}+ k_4^{i_1}= 0$. As in \nameref{para:unique_k_case1} above, this implies that  $k_3^{i_1}+k_4^{i_1} = \tilde k_3^{i_1}+\tilde k_4^{i_1}$. Further, choosing two indices $j_2,j_3 \in \{1,\ldots,p\}$ such that $j_1,j_2,j_3$ are pairwise distinct, it follows that $ \mu_{j_2} + k_3^{i_1} + k_4^{i_1}\neq 0$ and $ \mu_{j_3} + k_3^{i_1}+  k_4^{i_1} \neq  0$ by the $(\mu_j)_j$ being different. Using \eqref{eq:coefficient_basic_equality} and $k_3^{i_1}+ k_4^{i_1} = \tilde k_3^{i_1}+ \tilde k_4^{i_1}$ this implies
		\begin{align}
			\label{eq:klambda}
		K_1^{i_1}\lambda_{j_2} = \tilde K_1^{i_1} \tilde \lambda_{j_2}\qquad \text{and} \qquad K_1^{i_1}\lambda_{j_3} = \tilde K_1^{i_1} \tilde \lambda_{j_3}.
		\end{align}
		Using that the $K^{i_1}_1, \tilde K^{i_1}_1$ are nonzero (and hence, also $\tilde\lambda_{j_2},\tilde\lambda_{j_3}\neq 0$ as $\lambda_{j_2},\lambda_{j_3}\neq 0$), \eqref{eq:klambda} implies
		\[ \frac{\tilde \lambda_{j_2}}{\lambda_{j_2}} = \frac{\tilde \lambda_{j_3}}{\lambda_{j_3}}\neq 0.
		\]
		Combining this with the equations \eqref{eq:coefficient_basic_equality} for
		$i=i_0$ and $j=j_2,j_3$ we obtain
		\[
		\frac{\mu_{j_3}+\tilde k_3^{i_0}+\tilde k_4^{i_0}}{\mu_{j_3}+k_3^{i_0}+ k_4^{i_0}} = \frac{\mu_{j_2}+\tilde k_3^{i_0}+\tilde k_4^{i_0}}{\mu_{j_2}+k_3^{i_0}+ k_4^{i_0}}.
		\]
		Reformulating this equation and using that $\mu_{j_2} \neq \mu_{j_3}$ this
		implies that $k_3^{i_0}+ k_4^{i_0} = \tilde k_3^{i_0}+ \tilde k_4^{i_0}$. By analogous arguments as in \nameref{para:unique_k_case1} we derive $k_2^{i_0} =
		\tilde k_2^{i_0}$, $k_3^{i_0} =
		\tilde k_3^{i_0}$ and $k_4^{i_0} =
		\tilde k_4^{i_0}$.
		
		\noindent \textbf{Case II.B} Assume that $ \mu_{j} + k_3^{i_1}+k_4^{i_1} \neq 0$ for all $j$. Defining $\Lambda_j := \tilde{\lambda}_j /\lambda_j\neq 0,  k_{34}^{i_s}:=k_3^{i_s}+k_4^{i_s}, \tilde k_{34}^{i_s}:=\tilde k_3^{i_s}+\tilde k_4^{i_s}$, we then obtain from \eqref{eq:coefficient_basic_equality} for pairwise distinct $j_1,j_2,j_3\in\{1,\dots, p\}$ that
		\begin{align}
			\Lambda_{j_1} \frac{\mu_{j_1}+\tilde k_{34}^{i_s}}{\mu_{j_1}+k_{34}^{i_s}} = \Lambda_{j_2} \frac{\mu_{j_2}+\tilde k_{34}^{i_s}}{\mu_{j_2}+k_{34}^{i_s}} = \Lambda_{j_3} \frac{\mu_{j_3}+\tilde k_{34}^{i_s}}{\mu_{j_3}+k_{34}^{i_s}}.\notag
		\end{align}
		for $s=0,1$. From this, we conclude that
		
		\begin{equation} \label{eq:coeff_eq_inter}
			0= \frac{\mu_{j_r}+\tilde k_{34}^{i_0}}{\mu_{j_r}+ k_{34}^{i_0}}\frac{\mu_{j_s}+\tilde k_{34}^{i_1}}{\mu_{j_s}+ k_{34}^{i_1}}
			-  \frac{\mu_{j_r}+\tilde k_{34}^{i_1}}{\mu_{j_r}+ k_{34}^{i_1}}\frac{\mu_{j_s}+\tilde k_{34}^{i_0}}{\mu_{j_s}+ k_{34}^{i_0}}
		\end{equation}
		for $r,s \in \{1,2,3\}$ with $r\neq s$. Multiplying \eqref{eq:coeff_eq_inter} with the denominator $(\mu_{j_r}+
		k_{34}^{i_0})(\mu_{j_s}+ k_{34}^{i_1})(\mu_{j_r}+ k_{34}^{i_1})(\mu_{j_s}+ k_{34}^{i_0})$
		and further dividing by $\mu_{j_r} - \mu_{j_s}\neq 0$ we obtain
		\begin{multline*}
			0 = \mu_{j_r}\mu_{j_s}\left(\tilde k_{34}^{i_0}-\tilde k_{34}^{i_1} +k_{34}^{i_1}-k_{34}^{i_0}\right)+\left(\mu_{j_r}+\mu_{j_s}\right)\left(k_{34}^{i_1}\tilde k_{34}^{i_0 }-k_{34}^{i_0}\tilde k_{34}^{i_1}\right)\\+\left(k_{34}^{i_1}-k_{34}^{i_0}\right)\tilde k_{34}^{i_0}\tilde k_{34}^{i_1}+\left(\tilde k_{34}^{i_0}-\tilde k_{34}^{i_1}\right)k_{34}^{i_0} k_{34}^{i_1}=: Q^0_{r,s}
		\end{multline*}
		for $r,s \in \{1,2,3\}$ with $r\neq s$. Dividing $Q^0_{r,s}-Q^0_{r,t}$ for $r,s,t\in \left\{1,2,3\right\}$ with $s\neq t$ by $\mu_{j_s} - \mu_{j_t}\neq 0$ gives the identity
		\begin{equation} \label{eq:coeff_intermediate_equation}
			0 = \mu_{j_r}\left(\tilde k_{34}^{i_0}-\tilde k_{34}^{i_1}+k_{34}^{i_1}-k_{34}^{i_0}\right)+\left(k_{34}^{i_1}\tilde k_{34}^{i_0}-k_{34}^{i_0}\tilde k_{34}^{i_1}\right)
		\end{equation}
		which is a linear equation in $\mu_{j_r}$ with more than one distinct root (in fact $\mu_{j_1},\mu_{j_2}$ and $\mu_{j_3}$ are roots). As a consequence,
		\[
		\tilde k_{34}^{i_0}-k_{34}^{i_0} = \tilde k_{34}^{i_1} - k_{34}^{i_1},
		\]
		i.e., $\tilde k_{34}^{i_0} = k_{34}^{i_0} + \epsilon $ and $\tilde k_{34}^{i_1} =
		k_{34}^{i_1} + \epsilon $ for some $\epsilon\in \R$. Inserting this into
		\eqref{eq:coeff_intermediate_equation} gives
		\[ \epsilon(k_{34}^{i_1} - k_{34}^{i_0}) = 0  \] which, together with $k_{34}^{i_1} \neq k_{34}^{i_0}$, yields $\epsilon = 0$ and
		in particular $k_3^{i_0}+k_4^{i_0}=k_{34}^{i_0} = \tilde{k}_{34}^{i_0}=\tilde k_3^{i_0}+\tilde k_4^{i_0}$ as desired. By analogous arguments as in \nameref{para:unique_k_case1} we derive $k_2^{i_0} =
		\tilde k_2^{i_0}$, $k_3^{i_0} =
		\tilde k_3^{i_0}$ and $k_4^{i_0} =
		\tilde k_4^{i_0}$.

		\paragraph{Uniqueness of the remaining $k_2^i,k_3^i, k_4^i$ and of the $K_1^i$ and $(\lambda_j)_j$ up to a constant factor.} 
		Take $i_0 $ to be a region where $\tilde k_2^{i_0} = k_2^{i_0}, \tilde k_3^{i_0}=k_3^{i_0}, \tilde k_4^{i_0}=k_4^{i_0}$. We know already that three such regions $i_s$ exist by the considerations on \nameref{para:unique_3k}, with pairwise distinct $k_3^{i_s}+k_4^{i_s}$ for $s\in \left\{1,2,3\right\}$ by the considerations on \nameref{para:unique_3alpha}. Then by \eqref{eq:coefficient_basic_equality} for each $j$ with $\mu_j+k_3^{i_0}+k_4^{i_0}\neq 0$ it follows 
		\begin{equation} \label{eq:coefficient_equality_interestimate}
			K^{i_0}_1 \lambda_j = \tilde K^{i_0}_1 \tilde \lambda_j.
		\end{equation}
		 Since at most one $\hat{j}$ exists with $\mu_{\hat{j}}+k_3^{i_0}+k_4^{i_0}=0$, it follows with $\zeta := K^{i_0}_1 / \tilde  K^{i_0}_1\neq 0$, that $\tilde \lambda_j = \zeta  \lambda_j$ for all $j\neq \hat{j}$. To see $\tilde \lambda_{\hat{j}}=\zeta\lambda_{\hat{j}}$ note that by \eqref{eq:coefficient_basic_equality} we have $K_1^{i_1}\lambda_{\hat{j}}=\tilde K_1^{i_1}\tilde \lambda_{\hat{j}}$ and $K_1^{i_1}\lambda_{j}=\tilde K_1^{i_1}\tilde \lambda_{j}$ for at least one $j\neq \hat{j}$ by $p\geq 4$ which already fulfills \eqref{eq:coefficient_equality_interestimate}. Thus, we derive $\tilde K_1^{i_1}/K_1^{i_1}=\tilde K_1^{i_0}/K_1^{i_0}$ and as claimed $\tilde \lambda_{\hat{j}}=\zeta\lambda_{\hat{j}}$. 
		 
		 We now aim to show that, for all $i \in \{1,\ldots,n\}$, $k_2^i = \tilde k_2^i$, $k_3^i = \tilde k_3^i$, $k_4^i = \tilde k_4^i$  and $K_1^i = \zeta \tilde K_1^i$.
		
		Consider $i \in \{1,\ldots,n\}$ fixed. To simplify notation, we drop here the
		index $i$, e.g., we write $K_1 = K_1 ^i$, $k_2 = k_2^i$, $k_3 = k_3^i$ and $k_4 = k_4^i$ and similar for $\tilde K_1, \tilde k_2, \tilde k_3, \tilde k_4$ and $\alpha_1, \alpha_2, \tilde\alpha_1, \tilde\alpha_2$. Consider now three cases:\vspace{-0.4cm}
		
		\paragraph{Case I.}\label{para:last_case1} There exist $j_0\neq j_1$ with $\left\{\alpha_1,\alpha_2\right\}=\left\{\mu_{j_0},\mu_{j_1}\right\}$. As a consequence, the coefficients of $te^{\alpha_1 t}$ and $te^{\alpha_2t}$ in \eqref{ctextended_mu_nonzero} are nonzero and by the unique interpolation result in Lemma \ref{lem:polyexpinter} also the coefficients of $te^{\tilde\alpha_1 t}$ and $te^{\tilde\alpha_2t}$ are nonzero, implying $\left\{\alpha_1, \alpha_2\right\}=\left\{\tilde \alpha_1, \tilde \alpha_2\right\}$ and by Lemma \ref{cor:alphas} that $\alpha_1=\tilde\alpha_1$ and $\alpha_2=\tilde\alpha_2$. We know already from the
		considerations on \nameref{para:unique_3k} that consequently, $k_2 = \tilde k_2$, $k_3 = \tilde k_3$ and $k_4 = \tilde k_4$, such that, from \eqref{eq:coefficient_basic_equality} for some $j\in\left\{1,\ldots,p\right\}$ with $\mu_j+k_3+k_4\neq0$,
		we get
		\[
		K_1 \lambda_j =  \tilde K_1 \tilde  \lambda_j= \tilde K_1 \zeta  \lambda_j,
		\]
		and also $K_1 = \zeta \tilde K_1$ as desired since $\lambda_j\neq 0$.
		
		\noindent\textbf{Case II.} It holds $\mu_j \notin \left\{\alpha_1, \alpha_2\right\}$ for all $j$.  Equating the coefficients in the representations of $C_\tis$ and $\tilde C_\tis$ (see \eqref{coeff3}), and
		using $\tilde \lambda_j = \zeta \lambda_j$ for $j\in\left\{1,\ldots,p\right\}$, we get that
		\begin{equation} \label{eq:remaining_coefficients_basic_equation}
			\frac{\tilde K_1\zeta  (\mu_j +\tilde  k_3+\tilde k_4)}{(\mu_j-\tilde\alpha_1)(\mu_j-\tilde \alpha_2)} =
			\frac{K_1(\mu_j + k_3+k_4)}{(\mu_j-\alpha_1)(\mu_j-\alpha_2)}
		\end{equation}
		for $j=1,\ldots,p$. Now we show that, from \eqref{eq:remaining_coefficients_basic_equation}, it follows that $\zeta \tilde K_1 = K_1$, $\tilde k_2 = k_2$, $ \tilde k_3 = k_3$ and $\tilde k_4 = k_4$. For this, we again need to distinguish several cases.\vspace{-0.4cm}

		\paragraph{Case II.A.}\label{para:last_caseIIa} $\tilde{k}_3+\tilde k_4+\mu_{j_0}=0$  for at least one $j_0 \in \{1,2,3,4\}$. This implies that also $k_3 +k_4+ \mu_{j_0} =0$ by \eqref{eq:remaining_coefficients_basic_equation} and hence that $\tilde k_3+\tilde k_4 = k_3+k_4$. Considering $j_s \in \{1,2,3,4\} \setminus \{j_0\}$ for $s\in\left\{1,2,3\right\}$ it follows from the $\mu_j$ being pairwise distinct that $k_3 +k_4+ \mu_{j_s} \neq 0$ for $s\in\left\{1,2,3\right\}$, which implies that also $\tilde k_3  +\tilde k_4+ \mu_{j_s} \neq 0$ for $s\in\left\{1,2,3\right\}$ and, consequently, by \eqref{eq:remaining_coefficients_basic_equation}, the identity $k_3+k_4 = \tilde k_3 +\tilde k_4$ and by the two identities in \eqref{alpha_identities} of Lemma \ref{cor:alphas} that
		\begin{equation}
			\label{zeta_tilde_K1_representation}
			\zeta  \tilde K_1/K_1 =\frac{\mu_{j_s}^2+\mu_{j_s}(\tilde k_2+k_3+k_4)+\tilde k_2\tilde k_4}{\mu_{j_s}^2+\mu_{j_s}(k_2+k_3+k_4)+k_2k_4}
		\end{equation}
		for $s\in\left\{1,2,3\right\}$. Setting the identity \eqref{zeta_tilde_K1_representation} equal for $1\leq r,s\leq 3$ with $r\neq s$, rearranging the terms and dividing by $\mu_{j_r}-\mu_{j_s}\neq 0$ we derive
		\begin{equation}
			\label{eq:mu_r_s}
			Q_{r,s}^1:=\mu_{j_r}\mu_{j_s}(\tilde k_2-k_2)+(\mu_{j_r}+\mu_{j_s}+k_3+k_4)(\tilde k_2\tilde k_4-k_2k_4)+k_2\tilde k_2(\tilde k_4-k_4)=0.
		\end{equation}
		Dividing $Q^1_{r,s}-Q^1_{r,t}$ for $r,s,t\in \left\{1,2,3\right\}$ with $s\neq t$ by $\mu_{j_s} - \mu_{j_t}\neq 0$ we obtain the identity
		\begin{equation}
			\label{eq:mu_1}
			\mu_{j_r}(\tilde k_2-k_2)+(\tilde k_2\tilde k_4-k_2k_4)=0
		\end{equation}
		which is a linear equation in $\mu_{j_r}$ with more than one distinct root (in fact $\mu_{j_1},\mu_{j_2}$ and $\mu_{j_3}$ are roots). Hence, we recover $\tilde k_2 = k_2\neq 0$ and finally, also $\tilde k_4=k_4$ by \eqref{eq:mu_1}. Due to $\tilde k_3+\tilde k_4 = k_3+k_4$ also $\tilde k_3=k_3$ holds true. By inserting those identities in \eqref{zeta_tilde_K1_representation} we deduce $\zeta \tilde K_1=K_1$.\vspace{-0.4cm}
		
		\paragraph{Case II.B.}\label{para:last_caseIIb} $\tilde k_3 +\tilde k_4+ \mu_j \neq 0$ for all $j\in\left\{1,2,3,4\right\}$. 
		In this case we reformulate \eqref{eq:remaining_coefficients_basic_equation} under abuse of notation defining $k_{34}=k_3+k_4, k_{234}=k_2+k_3+k_4$ (and similarly $\tilde k_{34}, \tilde k_{234}$) to obtain
		\begin{equation} \label{eq:remaining_coefficients_reformulated_equation}
			\zeta \tilde K_1/K_1  = \frac{\mu_j^3+\mu_j^2(k_{34}+\tilde k_{234})+\mu_j(\tilde k_2\tilde k_4+k_{34}\tilde k_{234})+\tilde k_2\tilde k_4 k_{34}}{\mu_j^3+\mu_j^2(\tilde k_{34}+ k_{234})+\mu_j( k_2 k_4+\tilde k_{34} k_{234})+ k_2 k_4 \tilde k_{34}}
		\end{equation}
		for all $1\leq j\leq 4$. Setting the identities in \eqref{eq:remaining_coefficients_reformulated_equation} equal for $1\leq r,s\leq 4$ with $r\neq s$, rearranging the terms and dividing by $\mu_s-\mu_r \neq 0$ yields
		\begin{multline}
			\label{eq:chain1}
			Q_{r,s}^2:=\mu_r^2\mu_s^2(\tilde k_2-k_2)+\mu_r\mu_s(\mu_r+\mu_s)[\tilde k_2(k_{34}+\tilde k_4)-k_2(\tilde k_{34}+k_4)]\\
			+\mu_r\mu_s[(\tilde k_2\tilde k_4+k_{34}\tilde k_{234})(\tilde k_{34}+k_{234})-(k_2k_4+\tilde k_{34}k_{234})(k_{34}+\tilde k_{234})]\\
			+(\mu_s+\mu_r)[\tilde k_2\tilde k_4k_{34}(\tilde k_{34}+k_{234})-k_2k_4\tilde k_{34}(k_{34}+\tilde k_{234})]\\
			+[\tilde k_2\tilde k_4k_{34}(k_2k_4+\tilde k_{34}k_{234})-k_2k_4\tilde k_{34}(\tilde k_2\tilde k_4+k_{34}\tilde k_{234})]\\
			+(\mu_s^2+\mu_s\mu_r+\mu_r^2)(\tilde k_2\tilde k_4k_{34}-k_2k_4\tilde k_{34})=0.
		\end{multline}
		Dividing $Q^2_{r,s}-Q^2_{r,t}$ for $r,s,t\in \left\{1,2,3,4\right\}$ with $s\neq t$ by $\mu_{j_s} - \mu_{j_t}\neq 0$ we derive
		\begin{multline}
			\label{eq:chain2}
			Q_{r,s,t}^3:=\mu_r^2(\mu_s+\mu_t)(\tilde k_2-k_2)+\mu_r(\mu_r+\mu_s+\mu_t)[\tilde k_2(k_{34}+\tilde k_4)-k_2(\tilde k_{34}+k_4)]\\
			+\mu_r[(\tilde k_2\tilde k_4+k_{34}\tilde k_{234})(\tilde k_{34}+k_{234})-(k_2k_4+\tilde k_{34}k_{234})(k_{34}+\tilde k_{234})]\\
			+[\tilde k_2\tilde k_4k_{34}(\tilde k_{34}+k_{234})-k_2k_4\tilde k_{34}(k_{34}+\tilde k_{234})]\\
			+(\mu_r+\mu_s+\mu_t)(\tilde k_2\tilde k_4k_{34}-k_2k_4\tilde k_{34})=0.
		\end{multline}
		Dividing $Q^3_{r,s,t}-Q^3_{r,s,u}$ for $r,s,t,u\in \left\{1,2,3,4\right\}$ with $t\neq u$ by $\mu_{j_t} - \mu_{j_u}\neq 0$ we derive
		\[
			\mu_r^2(\tilde k_2-k_2)+\mu_r[\tilde k_2(k_{34}+\tilde k_4)-k_2(\tilde k_{34}+k_4)]+[\tilde k_2\tilde k_4k_{34}-k_2k_4\tilde k_{34}]=0
		\]
		which is a quadratic equation in $\mu_{j_r}$ with more than two distinct root (in fact $\mu_{j_1},\mu_{j_2}, \mu_{j_3}$ and $\mu_{j_4}$ are roots). Hence, $\tilde k_2 = k_2\neq0$ (from the quadratic coefficient), using this, also $\tilde k_3=k_3\neq 0$ (from the linear coefficient) and as a consequence, we obtain $\tilde k_4 = k_4$ (from the constant coefficient). Thus, as previously also $\zeta \tilde K_1 = K_1$ holds true.
		
		\noindent\textbf{Case III.} It remains to analyze w.l.o.g. the case that there exists some $j_1$ such that $\alpha_1=\mu_{j_1}$ and $\alpha_2\neq\mu_j$ for all $j$. By the unique interpolation result in Lemma \ref{lem:polyexpinter} it either holds that $\tilde \alpha_1 = \alpha_1$ or $\tilde \alpha_2 = \alpha_1$. W.l.o.g. assume that $\alpha_1 = \tilde\alpha_1$ and $\alpha_2\neq \tilde\alpha_2$. Note that the last inequality can be assumed w.l.o.g. since otherwise we fall into \nameref{para:last_case1}. Furthermore, note that the case $\alpha_1 = \tilde\alpha_2$ and $\alpha_2\neq \tilde\alpha_1$ can be similarly dealt with as below.
		
		\noindent\textbf{Case III.A.} As in \nameref{para:last_caseIIa}, first assume that $k_3+ k_4+\mu_{j_0}=0$ for some $j_0$, implying that $\tilde k_3+\tilde k_4 = k_3+k_4$.  Then \eqref{eq:remaining_coefficients_basic_equation} holds for at least two different $j_2\neq j_3$ with $j_0\notin\left\{j_2,j_3\right\}$ and simplifies to
		\begin{align}
			\label{eq:simplified_coefficient_representation}
			\zeta \tilde K_1/K_1 = \frac{(\mu_j-\tilde\alpha_2)(\mu_j+k_3+k_4)}{(\mu_j-\alpha_2)(\mu_j+\tilde k_3+ \tilde k_4)}
		\end{align}
		for $j\in\left\{j_2,j_3\right\}$. Setting the expression \eqref{eq:simplified_coefficient_representation} equal for $j_2$ and $j_3$ and simplifying the resulting term using $\tilde k_3+\tilde k_4 = k_3+k_4$ yields that
		\[
			0 = (\alpha_2-\tilde\alpha_2)(\mu_{j_2}+k_3+k_4)(\mu_{j_3}+k_3+k_4)
		\]
		which is a contradiction since each factor is nonzero (as $j_0\notin\left\{j_2,j_3\right\}$).
		
		\noindent\textbf{Case III.B.} Now as in \nameref{para:last_caseIIb} above, assume that $\tilde k_3+\tilde k_4+\mu_j\neq 0$ for all $j\in\left\{1,2,3,4\right\}$ and $\alpha_1=\tilde\alpha_1, \alpha_2\neq \tilde \alpha_2$ (the case $\alpha_1 = \tilde\alpha_2$ and $\alpha_2\neq \tilde\alpha_1$ can be similarly dealt with as below) we obtain by analogous steps but starting from the simplified representation \eqref{eq:simplified_coefficient_representation} that
		\begin{align}
			\label{eq:lastlineq}
			\mu_r[(\alpha_2+k_{34})-(\tilde\alpha_2+\tilde k_{34})]+[\alpha_2\tilde k_{34}-\tilde \alpha_2k_{34}]=0
		\end{align}
		attains three different roots $\mu_r$. Hence, using Lemma \ref{cor:alphas} and $\alpha_1=\tilde \alpha_1$, the linear coefficient of the expression \eqref{eq:lastlineq} being zero yields that $k_2=\tilde k_2$. Employing this identity and once more Lemma \ref{cor:alphas} implies, by the constant coefficient of \eqref{eq:lastlineq} being zero, that
		\[
			(\alpha_2-\tilde\alpha_2)(\alpha_1+k_2)=0
		\]
		which is a contradiction as $\alpha_2\neq\tilde\alpha_2$ by assumption and $\alpha_1+k_2\neq 0$ by Lemma \ref{cor:alphas}.\\
		
		As a consequence,  the remaining $\zeta \tilde K_1^i,\tilde k_2^i,\tilde k_3^i$ considered in this final part of the proof are uniquely determined as $\zeta \tilde K_1^i = K_1^i, \tilde k_2^i = k_2^i$ and $\tilde k_3^i = k_3^i$.
	\end{proof}

	As a consequence, knowledge of the tracer concentration in tissue $C_\tis^i(t_l)$ for $i=1,\ldots,n$ and for sufficiently many distinct time-points $t_l$, suffices to determine the coefficients $k_2^i,k_3^i, k_4^i$ uniquely and the coefficients $K_1^i$ uniquely up to a constant. In view of ODE system \eqref{eq:odesystem}, the ambiguity in $K_1^i$ cannot be improved without any knowledge of $C_\art$ (since one can always divide all $K_1^i$ by a nonzero constant and multiply $C_\art$ by the same constant). In fact, a single, nonzero measurement of $C_\art$ suffices, since for $\hat{C}_\art (\hat{s})\neq
	0$ the ground truth value at some time-point $\hat{s}$, the equality $C_\art(\hat{s}) = \hat{C}_\art
	(\hat{s}) = \tilde C_\art(\hat{s}) $ together with Proposition
	\ref{prop:main_uniqueness_result}, immediately imply that $\zeta=1$ such that also the $K_1^i$ are uniquely determined.
	
	In contrast to measurements of $C_\art$, used for parameter identification (see \cite{Ver13}), which are expensive in practice, it is much simpler to obtain measurements of the blood tracer concentration $C_\tot$, where $C_\art = f C_\tot$ with some unknown function $f$. Measurements of $C_\tot$ are sufficient to uniquely identify the $K_1^i$, provided sufficiently many measurements of $C_\tot$ in relation to a parametrization of $f$ are available. For that, we need the following notion.
	
	\begin{definition}[Parametrized function class for $f(t)$] For any $q \in \N$, a set of functions $F_q \subset \{f:\R \rightarrow \R\} $ is a degree-$q$ parametrized set if for any $f,\tilde{f}\in F_q$ and $\lambda \in \R$ it holds that $\lambda f - \tilde{f}$ attaining zero at $q$ distinct points implies that $\lambda=1$ and $f = \tilde{f}$.
	\end{definition}
	Simple examples are polynomials $f$ of degree $q-1$ that for some fixed $x_0,c\in\mathbb{R}$ with $c\neq 0$ satisfy $f(x_0) = c$ or polyexponential functions $f$ of degree $q/2$ (if $q$ is even) that for some fixed $x_0,c\in\mathbb{R}$ with $c\neq 0$ satisfy $f(x_0) = c$. The latter is frequently used in practice with $f(0)=1$ (see \cite{Ver13}).
	
	\begin{proposition} \label{prop:uniqueness_attenuation_maps}
		In the situation of Proposition \ref{prop:main_uniqueness_result}, assume in addition that $f,\tilde{f}:\R \rightarrow \R$ are
		functions contained in the same degree-$q$ parametrized set of functions, and are such that
		\[ C_\art(s_l) = f(s_l)C_\tot(s_l) \text{ and } \tilde C_\art(s_l) = \tilde f(s_l)C_\tot(s_l) \text{ for }l=1,\ldots,q,\]
		with $s_1,s_2,\ldots,s_q$ being $q$ different time points, and $C_\tot(s_l)\neq 0$
		given for $l=1,\ldots,q$. Then, all assertions of Proposition
		\ref{prop:main_uniqueness_result} hold with $\zeta=1$, and further
		\[ f = \tilde{f} .\]
	\end{proposition}
	\begin{proof}
		By Proposition \ref{prop:main_uniqueness_result} it holds $\tilde{C}_\art = \zeta C_\art$. Using that, by assumption,
		\[   \zeta f(s_{l})C_\tot(s_{l})  = \zeta C_\art(s_{ l}) = \tilde C_\art(s_{l}) = \tilde f(s_{ l})C_\tot(s_{ l}),\]
		we derive $(\zeta f - \tilde{f})(s_l) = 0$ for $l=1,\ldots,q$. Since
		$f,\tilde{f}:\R \rightarrow \R$ are functions contained in the
		same degree-$q$ parametrized set, this implies that $\zeta = 1$ and $f =
		\tilde{f}$ as claimed.
	\end{proof}

	Finally, a summary of the previous results is given as follows.
	\begin{theorem} \label{thm:main_uniqueness} Let $(p,n,((\lambda_j,\mu_j))_{j=1}^p ,((K_1^i,k_2^i,k_3^i,k_4^i))_{i=1}^n,(C^i_\tis)_{i=1}^n,C_\art)$
		be a ground-truth \\ configuration of the reversible two tissue compartment model such that
		\begin{enumerate}
			\item $p \geq 4$, $n\geq 3$ and  $K_1^i,k_2^i,k_3^i,k_4^i > 0$ for all $i=1,\ldots,n$,
			\item There are at least $p+3$ regions $i_1,\ldots,i_{p+3}$ where each the
			$k_3^{i_s}+k_4^{i_s}$ and the $\alpha_1^{i_s}, \alpha_2^{i_s}$ are pairwise distinct for
			$s=1,\ldots,p+3$.
		\end{enumerate}
		Let further $C_\tot:[0,\infty) \rightarrow [0,\infty)$ be the ground truth arterial whole blood tracer concentration.
		
		Then, for any other parameter configuration $(\tilde p,n,((\tilde
		\lambda_j,\tilde \mu_j))_{j=1}^{\tilde p} ,((\tilde K_1^i,\tilde k_2^i,\tilde
		k_3^i, \tilde k_4^i))_{i=1}^n,(\tilde C^i_\tis)_{i=1}^n,$ $\tilde C_\art)$ such that the
		conditions 1) and 2) above also hold, it follows from
		\[C_\tis(t_l) = \tilde C_\tis(t_l) \quad \text{for }l=1,\ldots, T\]
		with $T \geq \max\{ 2(p+4),2(\tilde p + 4)\}$ and the $t_1,\ldots,t_T$ pairwise
		distinct, that,
		\[  K_1^i = \zeta \tilde K_1^i,  k_2^i =\tilde k_2^i, k_3^i =\tilde k_3^i \text{ and }  k_4^i =\tilde k_4^i\text{ for all }i=1,\ldots,n,
		\]
		for some constant $\zeta\neq 0$, that $p= \tilde{p}$, and that (up to re-indexing)
		\[ \tilde \mu_j = \mu_j \text{ and } \tilde \lambda_j =\zeta  \lambda_j \text{ for all }i=1,\ldots,p.
		\]
		If further $f:[0,\infty) \rightarrow [0,\infty)$ is a ground-truth ratio between $C_\art$ and $C_\tot$
		in a degree-$q$ \\ parametrized set of functions and $\tilde
		f:[0,\infty) \rightarrow [0,\infty)$ is a function in the same
		degree-$q$ parametrized set of functions such that
		\[ C_\art(s_l) = f(s_l)C_{\tot}(s_l) \text{ and } \tilde C_\art(s_l) = \tilde f(s_l)C_{\tot}(s_l) \text{ for }l=1,\ldots,q,\]
		with the $s_1,\ldots,s_q$ pairwise distinct and $C_{\tot}(s_l)\neq 0$ given,
		then $\zeta=1$ and
		\[ f = \tilde{f}.
		\]
		\begin{proof}
			This follows immediately by Lemma \ref{lem:simpler_version_of_main_assumption} and Proposition \ref{prop:main_uniqueness_result}. In fact, Lemma \ref{lem:simpler_version_of_main_assumption} ensures that the assumptions of Proposition \ref{prop:main_uniqueness_result} are satisfied provided that 1.) and 2.) hold. In case $\tilde{p}\leq p$ the result immediately follows from Propositions \ref{prop:main_uniqueness_result} and \ref{prop:uniqueness_attenuation_maps}. In case $\tilde{p}> p$ it follows from interchanging the roles of the two configurations and again applying Propositions \ref{prop:main_uniqueness_result} and \ref{prop:uniqueness_attenuation_maps}.
		\end{proof}
	\end{theorem}
	\begin{remark}[Practical interpretation]
		The result in Theorem \ref{thm:main_uniqueness} can be interpreted as follows: i) Given that the assumptions of Theorem \ref{thm:main_uniqueness} hold for the ground truth parameter configuration (which, as argued in Remark \ref{rem:probabilistic_interpretation} is true almost surely) one can check a-posteriori if a numerically computed configuration matching the measured data satisfied the assumptions of Theorem \ref{thm:main_uniqueness}. If this is true (which is most likely the case), one can be sure that the obtained parameter configuration is the ground truth parameter configuration.
		ii) Generally, image-based measurements of the tracer concentration in tissue (without knowing $C_\art$ or $C_\tot$) are sufficient to determine the $(k_2^i,k_3^i,k_4^i)_{i=1}^n$ and the $(K_1^i)_{i=1}^n$ up to a global constant. In case the function $f$ is modeled by a biexponential function, which is sufficient in practice, the $(K_1^i)_{i=1}^n$ can be identified on the basis of four measurements of $C_\tot$, for which image-based measurements or a simple blood analysis are sufficient.
	\end{remark}
	\begin{remark}[Nontrivial fractional blood volume]
		\label{rem:cpet}
		In practice a realistic generalization is to assume that the voxel measurements in the PET images are given by a mixture of the blood tracer and tissue concentration rather than entirely the latter, i.e., the voxel measurements fulfill
		$$C_\pet(t) = (1-\fbv)\cdot C_\tis(t)+\fbv\cdot C_\tot(t),$$ where $\fbv$ with $0\leq \fbv<1$ 
		describes the fractional blood volume. The setting of known $\fbv$ and $C_\tot$ (at the same time points as the PET image measurements) is covered by above results directly. 
		If both $\fbv$ and $C_\tot$ are not available, an ansatz would be to parametrize $C_\tot$ by a polyexponential function, and assume enough measurements of $C_\pet$ to be available for the unique interpolation result 
		of Lemma \ref{lem:polyexpinter} to be applicable. With this, similar techniques as in Proposition \ref{prop:main_uniqueness_result} can be applied. 
	\end{remark}
	\section{Conclusion}
	The central analytic result of this work is that most tracer tissue kinetic parameters of the reversible two tissue compartment model can be recovered from standard PET measurements based on mild assumptions, specifically if sufficiently many different regions are modeled and enough standard image-based PET measurements at different time points are available. Furthermore, in case sufficiently many measurements of the total arterial concentration are available (which can be obtained from image-based measurements or simple blood sample analysis), the full recovery of all tracer tissue kinetic parameters of the model is possible. The significance of the analytic result, which holds in the idealized noiseless regime, is that it verifies parameter identifiability using practically easily obtainable quantities from image-based measurements or with simple blood sampling in principle. While this kind of result has been already shown for the irreversible two tissue compartment model it is novel for its reversible extension, which is practically highly relevant in quantitative PET imaging. 

	An important future research direction, that we will take, is to numerically investigate tissue parameter identifiability based only on image-based PET measurements and estimations of the total arterial tracer concentration for the (ir)reversible two tissue compartment model for real measurement data. We believe that to apply the analytic results meaningfully in practice, it is essential to also study stability and model -uncertainty and error.
	
	\bibliographystyle{siamplain}
	\bibliography{references}
\end{document}